\documentclass[12pt]{article}
\usepackage{amsmath,amsthm,amssymb}
\usepackage{eucal}
\usepackage{mathrsfs}
\usepackage{esint}
\numberwithin{equation}{section}
\usepackage{graphicx,graphics}
\usepackage{pdfsync}
\usepackage{color}
\usepackage[a4paper,margin=2.70truecm]{geometry}
\newcounter{marnote}

\def \dis {\displaystyle}

\def \into {\int_\Omega}
\def \confai {-\kern -.5em\rightharpoonup}

\def \div{\mbox{\rm div}\,}

\def \into {\int_\Omega}
\def \al {\alpha}
\def \be {\beta}

\def \ep {\varepsilon}
\def \om {\omega}
\def \Om {\Omega}
\def \la {\lambda}

\def \NN {\mathbb N}

\def \RR {\mathbb R}

\def \E {\mathscr{E}}

\def \beq {\begin{equation}}
\def \eeq {\end{equation}}
\def \ba {\begin{array}}
\def \ea {\end{array}}

\def \ecart {\noalign{\medskip}}
\DeclareMathOperator{\sgn}{sgn}

\DeclareMathOperator{\dom}{dom}
\newtheorem{Thm}{Theorem}[section]

\newtheorem{Pro}[Thm]{Proposition}

\newtheorem{Adef}[Thm]{Definition}
\newenvironment{Def}{\begin{Adef}}{\end{Adef}}
\newtheorem{Arem}[Thm]{Remark}
\newenvironment{Rem}{\begin{Arem}\rm}{\end{Arem}}
\newtheorem{Aexa}[Thm]{Example}

\newtheorem{Anot}[Thm]{Notation}

\def \reff #1.{figure~\ref{#1}}
\def \refs #1.{Section~\ref{#1}}
\def \refss #1.{Subsection~\ref{#1}}
\def \refD #1.{Definition~\ref{#1}}
\def \refT #1.{Theorem~\ref{#1}}
\def \refL #1.{Lemma~\ref{#1}}
\def \refC #1.{Corollary~\ref{#1}}
\def \refP #1.{Proposition~\ref{#1}}
\def \refPt #1.{Properties~\ref{#1}}
\def \refR #1.{Remark~\ref{#1}}
\def \refE #1.{Example~\ref{#1}}
\def \refN #1.{Notation~\ref{#1}}
\title{On the regularity of optimal potentials in control problems governed by elliptic equations}

\begin{document}
\maketitle
\centerline{\large G. Buttazzo$^\dag$,\hskip 0.3cm J. Casado-D\'{\i}az$^{\dag\dag}$,\hskip 0.3cm F. Maestre$^{\dag\dag}$}
\bigskip 
\centerline{{$^\dag$} Dipartimento di Matematica, Universit\`a di Pisa,}
\centerline{Largo B. Pontecorvo, 5}
\centerline{56127 Pisa, ITALY}
\bigskip
\centerline{$^{\dag\dag}$ Dpto. de Ecuaciones Diferenciales y An\'alisis Num\'erico,}
\centerline{Facultad de Matem\'aticas, C. Tarfia s/n}
\centerline{41012 Sevilla, SPAIN}
\bigskip
\centerline{e-mail: giuseppe.buttazzo@unipi.it, jcasadod@us.es,
fmaestre@us.es}

\begin{abstract}
In this paper we consider optimal control problems where the control variable is a potential and the state equation is an elliptic partial differential equation of a Schr\"odinger type, governed by the Laplace operator. The cost functional involves the solution of the state equation and a penalization term for the control variable. While the existence of an optimal solution simply follows by the direct methods of the calculus of variations, the regularity of the optimal potential is a difficult question and under the general assumptions we consider, no better regularity than the $BV$ one can be expected. This happens in particular for the cases in which a bang-bang solution occurs, where optimal potentials are characteristic functions of a domain. We prove the $BV$ regularity of optimal solutions through a regularity result for PDEs. Some numerical simulations show the behavior of optimal potentials in some particular cases.
\end{abstract}


\medskip


\section{Introduction}\label{intro}

The starting point of our research is an optimal control problem of the form
\beq\label{cost}
\min\int_\Om\big[j(x,u)+\psi(m)\big]\,dx,
\eeq
governed by the state equation
\beq\label{state}
\begin{cases}
-\Delta u+mu=f\quad\text{in }\Om\\
u\in H^1_0(\Om).
\end{cases}
\eeq
Here $\Om$ is a bounded open subset of $\RR^N$, the control variable $m$ is assumed to be nonnegative, $f\in L^2(\Om)$, and $j,\psi$ are suitable integrands. We assume that $\psi$ has a superlinear growth, which automatically implies that the control variables are in $L^1(\Om)$. Notice that, when $j(x,u)=f(x)u$, the problem can be written in the variational form
$$\min\Big\{-2\E(m)+\Psi(m)\ :\ m\in L^1(\Om),\ m\ge0\Big\},$$
where
\[\begin{split}
&\Psi(m)=\int_\Om\psi(m)\,dx\\
&\E(m)=\min\left\{\int_\Om\Big[\frac12|\nabla u|^2+\frac12 mu^2-f(x)u\Big]\,dx\ :\ u\in H^1_0(\Om)\right\}.
\end{split}\]
In this case it is possible to see that the control variable $m$ can be eliminated, obtaining the {\it auxiliary} variational problem
\beq\label{aux}
\min\left\{\int_\Om\Big[|\nabla u|^2+\psi^*(u^2)-2f(x)u\Big]\,dx\ :\ u\in H^1_0(\Om)\right\},
\eeq
where $\psi^*$ denotes the Fenchel-Moreau conjugate of the function $\psi$. Setting $g(s)=s(\psi^*)'(s^2)$, the unique solution $\hat u$ of \eqref{aux} can be obtained through the PDE
\beq\label{auxpde}
\begin{cases}
-\Delta u+g(u)=f\quad\text{in }\Om\\
u\in H^1_0(\Om),
\end{cases}
\eeq
and the optimal control $\hat m$ can be then recovered as
$$\hat m=(\psi^*)'(\hat u^2).$$
For general integrands $j(x,u)$ the elimination procedure above is not possible, and the necessary conditions of optimality involve an adjoint state variable and the corresponding adjoint PDE. Nevertheless, we can show that the optimal control problem \eqref{cost}, \eqref{state} admits a solution $(\hat u,\hat m)$. The main goal of the present paper is to show that, under suitable assumptions on the data, the optimal control $\hat m$ has additional regularity properties. In particular, we show that $\hat m\in BV(\Om)$.

We stress that, under the general assumptions we consider on the function $\psi$, higher regularity properties on $\hat m$ do not hold. Indeed, when
$$\psi(s)=\begin{cases}
s&\text{if }s\in[\alpha,\beta]\qquad\text{(with $0\le\alpha<\beta$)}\\
+\infty&\text{otherwise,}
\end{cases}$$
the optimal control $\hat m$ is of {\it bang-bang} type, that is
$$\hat m=\alpha+(\beta-\alpha)1_E$$
for a suitable set $E$, which then turns out to be a set with finite perimeter.

Problems of this form arise for instance in some biological models, governed by logistic diffusive equations, where one aims to control the size of a total population or the optimal location of resources, see for instance \cite{MNP1} and \cite{MNP2}.

The proof of the $BV$ regularity above is obtained through a careful analysis of a nonlinear elliptic PDE of the form \eqref{auxpde}. In general, for the cases arising from optimal control problems, the right-hand side $f$ has a low summability and does not belong to the space $H^{-1}(\Om)$; the definition of solution is then more involved and has to be given as in the theory of renormalized solutions (see for instance \cite{BeBoGaGaPiVa}, \cite{DMMuOrPr}).

In Section \ref{Not} we list the notation that is used along the paper, in Section \ref{S1} we study the semilinear problems \eqref{auxpde} with the nonlinearity $g$ possibly discontinuous and the right-hand side $f$ having a low summability. Our goal is to show that, when $f\in BV(\Om)$, the solution $u$ is such that $g(u)\in BV(\Om)$. Section \ref{S2} deals with the application of the result above to the optimal control problem \eqref{cost}, \eqref{state}. In Section \ref{S5} we consider some relevant examples with various particular choices of the data. Finally in Section \ref{S6} we provide some numerical simulations which show the bang-bang behavior of optimal solutions in some cases, as well as the continuous behavior in some other ones.

\section{Notation}\label{Not}

In this section we fix the notation that we use in the rest of the paper.

\begin{itemize}
\item For $s\in\RR$, we denote by $\sgn(s)$ the sign of $s$.
\item For a measurable set $A\subset\RR^N$, we denote by $|A|$ its Lebesgue measure.
\item For a function $g:\RR\to\RR$, and $s\in\RR$, we denote by $g_-(s)$, and $g_+(s)$ the limits on the left and on the right in $s$, respectively (if they exist).
\item For $k>0$ we define the functions $T_k,S_k\in W^{1,\infty}(\RR)$ by
$$T_k(s)=\begin{cases}
s&\hbox{if }|s|<k\\
k\,\sgn(s)&\hbox{if }|s|\ge k,
\end{cases}\qquad
S_k(s)=\begin{cases}
1&\hbox{if }|s|\le k\\
\dis 2-{|s|\over k}&\hbox{if }k<|s|<2k\\
0&\hbox{if }|s|\ge2k.
\end{cases}$$
\item We denote by ${\cal M}(\Om)$ the space of Borel measures in $\Om$ with finite total mass.
\item If $\psi:\RR\to\RR\cup\{\infty\}$ is a lower semicontinuous proper convex function, we denote by $\dom(\psi)$ the domain of $\psi$, defined by
$$\dom(\psi)=\big\{s\in\RR\ :\ \psi(s)<\infty\big\}.$$
It is a non-empty interval of $\RR$, which we assume not reduced to a point. For $s\in\dom(\psi)$, we denote by $\partial\psi(s)$ the subdifferential of $\psi$ at $s$, defined by
\beq\label{defsd}
\partial\psi(s)=\big\{\tau\in\RR\ :\ \tau(t-s)\le\psi(t)-\psi(s)\ \text{ for every }t\in\dom(\psi)\big\}.\eeq
Taking into account that the map
\beq\label{apico}
(s,t)\mapsto{\psi(t)-\psi(s)\over t-s}\qquad\text{for }s<t\eeq
is non-decreasing in both variables $s$ and $t$, we have that for every $s\in\dom(\psi)$ the two limits
$$\begin{cases}
\dis d_+\psi(s):=\lim_{\ep\searrow0}{\psi(s+\ep)-\psi(s)\over\ep}\in(-\infty,\infty]\\
\dis d_-\psi(s):=\lim_{\ep\searrow 0}{\psi(s)-\psi(s-\ep)\over \ep}\in[-\infty,\infty)
\end{cases}$$
exist and
\beq\label{carsd}\partial \psi(s)=\big[d_-\psi(s),d_+\psi(s)\big].\eeq
We also recall that every convex function $\psi$ is locally Lipschitz in the interior of its domain.
\item We denote by $\psi^*$ the Fenchel-Moreau conjugate of $\psi$, defined by
$$\psi^*(t):=\sup_{s\in \dom(\psi)}\big\{ts-\psi(s)\big\}\qquad\forall\,t\in\RR.$$
We recall that $\psi^*$ is also a convex lower semicontinuous function and that
$$(\psi^*)^*=\psi.$$ 
Moreover, we have
\beq\label{repsp} t\in\partial\psi(s)\iff s\in\partial\psi^*(t)\iff ts=\psi(s)+\psi^*(t).\eeq
\end{itemize}

In the following we consider the integral functionals
\[\begin{split}
&J(u)=\int_\Om j(x,u)\,dx\qquad\text{defined for }u\in H^1_0(\Om);\\
&\Psi(m)=\int_\Om\psi(m)\,dx\qquad\text{defined for }m\in L^1(\Om),\ m\ge0;\\
&F(u,m)=J(u)+\Psi(m).
\end{split}\]
The function $\psi:[0,\infty)\to\RR\cup\{\infty\}$ is assumed lower semicontinuous proper and convex, and such that
\beq\label{psisl}\lim_{s\to\infty}\frac{\psi(s)}{s}=\infty;\eeq
in this way the functional $\Psi$ is well defined on $L^1(\Om)$ and $\Psi(m)<\infty$ implies that $m\in L^1(\Om)$. It is easy to see that, with the conditions above, the function $\psi$ is bounded from below; hence up to the addition of a constant, that does not modify our optimization problem, we may assume $\psi$ non-negative.

Concerning the integrand $j(x,s)$ we assume measurability in $x$, lower semicontinuity in $s$ and the bound
\beq\label{condj}
a(x)-c|s|^2\le j(x,s)
\eeq
for suitable $c\ge0$ and $a\in L^1(\Om)$.

The optimization problem we consider is then
$$\min\big\{F(u,m)\ :\ -\Delta u+mu=f,\ u\in H^1_0(\Om),\ m\in L^1(\Om)\big\},$$
where $f\in H^{-1}(\Om)$ is prescribed. In Section \ref{S2} we show that the optimal control problem above admits an optimal pair $(\hat u,\hat m)$. Under some additional assumptions, we obtain the corresponding necessary conditions of optimality and we study the regularity properties of $(\hat u,\hat m)$.


\section{Semilinear problems with a discontinuous term}\label{S1}

In the present section we are interested in the existence and regularity properties of the solutions of a semi-linear problem of the form
\beq\label{ecto0b}
\begin{cases}
-\Delta u+g(u)=f&\text{in }\Om\\
u=0&\text{on }\partial\Om,
\end{cases}\eeq
where the function $g$ is non-decreasing and not necessarily continuous. Moreover, we assume only low integrability on the right-hand side $f$. In particular, we do not assume $f\in H^{-1}(\Om)$, which leads us to work with renormalized or entropy solutions (see for instance \cite{BeBoGaGaPiVa}, \cite{BocGal}, \cite{CaCoVa}). In this sense, we introduce the following definition of solution for problem \eqref{ecto0b}.

\begin{Def}\label{defsol} Assume $f\in L^1(\Om)$. We say that a pair $(u,w)$ is a solution of \eqref{ecto0b} if it satisfies
\begin{align}
&\begin{cases}
u\in H^1_0(\Om)&\hbox{if }N=1\\
\dis u\in W^{1,p}_0(\Om)\ \forall\,p<{N\over N-1},\quad T_k(u)\in H^1_0(\Om)\ \forall\,k>0&\hbox{if }N\ge2,
\end{cases}\label{acsogT2}\\
&w\in L^1(\Om),\quad g_-(u)\leq w\leq g_+(u)\ \hbox{ a.e. in }\Om,\label{acsoRwT2}\\
&\begin{cases}
\dis\into\nabla u\cdot\nabla v\,dx+\into wv\,dx=\into fv\,dx\\
\forall\, v\in H^1_0(\Om)\cap L^\infty(\Om)\hbox{ such that } \exists k>0\ \hbox{ with } \nabla v=0\hbox{ a.e. in }\{|u|>k\}.
\end{cases}\label{Ecv}
\end{align}
\end{Def}

The existence and uniqueness of solutions for problem \eqref{ecto0b} is given by the following theorem.

\begin{Thm}\label{Th2ex} Let $\Om\subset\RR^N$ be a bounded open set, and $g:\RR\to\RR$ a non-decreasing function. Then, for every $f\in L^1(\Om)$, there exists a unique solution $(u,w)$ of \eqref{ecto0b} in the sense of Definition \ref{defsol}. Moreover, it satisfies
\begin{align}
&\|w-g_+(0)\|_{L^1(\Om)}\le\|f-g_+(0)\|_{L^1(\Om)},\label{acowT2}\\
&\begin{cases}
\dis\|u\|_{H^1_0(\Om)}\le C\|f\|_{L^1(\Om)}&\hbox{if }N=1,\\
\dis\|T_k(u)\|^2_{H^1_0(\Om)}\le \into |f-g_+(0)| |T_k(u)|\,dx\quad\forall\,k>0,\\
\dis\|u\|_{W^{1,p}_0(\Om)}\le C\|f-g_+(0)\|_{L^1(\Om)}\quad\forall\,p<{N\over N-1}&\hbox{if }N\ge2.
\end{cases}\label{acotuT2}
\end{align}
In these estimates the constant $C$ only depends on $|\Om|$, $p$, and $N$.\end{Thm}

\begin{Rem}\label{Retex} Assuming in Theorem \ref{Th2ex}, $f$ in $L^q(\Om)$, with $q>1$, the classical estimates for renormalized solutions (see for instance \cite{BeBoGaGaPiVa}, \cite{BocGal}, \cite{DMMuOrPr}) combined with Stampacchia's estimates (see \cite{Sta}) also prove that the solution $(u,w)$ of \eqref{ecto0b} satisfies
\beq\label{estrenoso}
\begin{cases}
\dis u\in W_0^{1,{Nq\over N-q}}(\Om)&\dis\hbox{if }1<q\leq {2N\over N+2}\\ \ecart
\dis u\in H^1_0(\Om)\cap L^{Nq\over N-2q}(\Om)&\dis\hbox{if }{2N\over N+2}<q<{N\over 2}\\ \ecart
\dis u\in H^1_0(\Om)\cap L^\infty(\Om)&\dis\hbox{if }{N\over2}<q.
\end{cases}\eeq
In particular, assuming $f\in L^{2N\over N+2}(\Om)$, we have that the solution $u$ in Theorem \ref{Th2ex} is in $H^1_0(\Om)$. In this case, it can be defined in a simpler way as the unique solution of the strictly convex minimum problem
$$\min_{v\in H^1_0(\Om)}\into \left({1\over 2}|\nabla v|^2+G(v)-fv\right)dx,$$
with $G:\RR\to\RR$ defined by
$$G(s)=\int_0^s g(r)\,dr\qquad \forall\,s\in \RR.$$
\end{Rem}

\begin{Rem}\label{Retex2}
Theorem \ref{Th2ex} and (\ref{estrenoso}) can be extended with the same proofs to the case where the operator $-\Delta u$ is replaced by $-\div(A(x)\nabla u)$ with $A$ a matrix function in $L^\infty(\Om)^{N\times N}$ satisfying the ellipticity condition
$$\exists\,c>0\hbox{ such that }\ A(x)\xi\cdot\xi\ge c|\xi|^2\quad\forall\,\xi\in\RR^N,\hbox{ a.e. }x\in\Om.$$
Moreover, the equation 
\beq\label{eqre}-\div(A(x)\nabla u)+w=f\eeq
is satisfied in the sense of distributions for $\Om$. Indeed,
in the case of the Laplacian operator, since $-\Delta $ is an isomorphism from $W^{1,p}_0(\Om)$ into $W^{-1,p}(\Om)$, for $1<p<\infty$, it is known that (\ref{Ecv}) is equivalent to require that $(u,w)$ satisfies \eqref{eqre} in the distributions sense in $\Om$. Thus, Theorem \ref{Th2ex} shows that for every $f\in L^1(\Om)$, there exists a unique solution $(u,w)\in W_0^{1,p}(\Om)\cap L^1(\Om)$, $1<p<N/(N-1)$, in the distributions sense, of
\beq\label{ecto0}
\begin{cases}-\Delta u+w=f&\text{in }\Om\\
u=0&\hbox{on }\partial\Om\\
g_-(u)\le w\le g_+(u)&\text{a.e. in }\Om.
\end{cases}\eeq
However, if we replace $-\Delta u$ by $-\div(A\nabla u)$ with $A$ as above, this is no longer true. We refer to \cite{Ser} for a classical counter-example to the uniqueness of the distributional solution.
\end{Rem}

The following result proves the continuous dependence with respect to the right-hand side and a maximum principle.

\begin{Thm}\label{Thunic} Let $\Om\subset\RR^N$ be a bounded open set and $g:\RR\to\RR$ a non-decreasing function. For $f_1,f_2\in L^1(\Om)$ we take $(u_1,w_1)$, $(u_2,w_2)$ solutions of (\ref{ecto0b}) with $f=f_1$ and $f=f_2$ respectively. Then, we have
\begin{align}
&\|w_1-w_2\|_{L^1(\Om)}\le\|f_1-f_2\|_{L^1(\Om)};\label{estdw}\\
&\|T_k(u_1-u_2)\|_{H^1_0(\Om)}\le k\|f_1-f_2\|_{L^1(\Om)}\qquad\forall\,k>0;\label{estdu1}\\
&\begin{cases}
\|u_1-u_2\|_{H^1_0(\Om)}\le C\|f_1-f_2\|_{L^1(\Om)}&\hbox{if }N=1\\
\dis\|u_1-u_2\|_{W^{1,p}_0(\Om)}\le C\|f_1-f_2\|_{L^1(\Om)}\quad\forall\,p\in\Big(1,{N\over N-1}\Big)&\hbox{if }N\ge2.\end{cases}\label{estdu2}
\end{align}
The constant $C>0$ in this last inequality only depends on $|\Om|$ if $N=1$. For $N\geq 2$ it only depends on $N$, $p$ and $|\Om|$. In addition,
$$f_1\le f_2\hbox{ a.e. in }\Om\ \Longrightarrow\ u_1\leq u_2\hbox{ a.e. in }\Om.$$
\end{Thm}

Our main result in this section proves that the function $w$ in Theorem \ref{Th2ex} is in $BV(\Om)$ when $f$ is in $BV(\Om)$. It will be used in Theorem \ref{Topr} to deduce some regularity results for the solution of a control problem governed by an elliptic equation, where the control variable corresponds to the coefficients of the zero order's term.

\begin{Thm}\label{Threg}
Let $\Om\subset\RR^N$ be a bounded open set of class $C^{1,1}$, and $g:\RR\to\RR$ a non-decreasing function. Then, for $f\in BV(\Om),$ the solution $(u,w)$ of \eqref{ecto0} is in $W^{2,{N\over N-1}}(\Om)\times BV(\Om)$. Moreover, there exists $C>0$ depending only on $\Om$ such that
\begin{align}
&\|u\|_{W^{2,{N\over N-1}}(\Om)}\leq C\big(\|f\|_{BV(\Om)}+|g_+(0)|\big),\label{acouTBV}\\
&\|\nabla w\|_{{\cal M}(\Om)^N}\leq C\big(\|f\|_{BV(\Om)}+|g_+(0)|).\label{acowTBV}
\end{align}
\end{Thm}
\par\medskip
Let us now prove Theorems \ref{Th2ex}, \ref{Thunic} and \ref{Threg}.

\begin{proof}[Proof of Theorem \ref{Th2ex}]
When $g$ is a continuous function, the result is well known from the theory of renormalized solutions for elliptic PDE (see for instance \cite{BeBoGaGaPiVa}, \cite{BocGal}, \cite{DMMuOrPr}). We recall how the corresponding estimates are obtained.

Taking $T_k(u)$, with $k>0$ as test function in (\ref{Ecv}) we have
\beq\label{e1Th2}
\int_{\{|u|<k\}}\hskip-6pt|\nabla u|^2\,dx+\into\big(g(u)-g(0)\big)T_k(u)\,dx=\into (f-g(0))T_k(u)\,dx.
\eeq
Thanks to the fact that $g$ is non-decreasing, we have $(g(u)-g(0))T_k(u)\geq 0$ a.e. in $\Om$ and this gives the second estimate in \eqref{acotuT2}. From this inequality, using the argument in Theorem 1 of \cite{BocGal}, we conclude that $u$ satisfies (\ref{acotuT2}).

Dividing by $k$ in (\ref{e1Th2}) and taking the limit as $k\to0$ we get
$$\lim_{k\to 0}{1\over k}\int_{\{|u|<k\}}\hskip-6pt |\nabla u|^2dx+\into |g(u)-g(0)|dx=\into (f-g(0))\sgn(u)\,dx.$$
This proves \eqref{acowT2} with $w=g(u)$.

Let us now prove the existence of solution for (\ref{ecto0}) in the case of $g$ just non-decreasing. For this purpose we replace $g$ by $g_n=g\ast \rho_n$, $n\in \NN$, with $\rho_n$ a sequence of mollifiers functions defined as
$$\rho_n(s)=n\rho\big(ns)\qquad\forall\,s\in\RR,$$
with 
$$\rho\in C^\infty(\RR),\quad \rho\geq 0\ \hbox{ in }\RR,\quad {\rm support}(\rho)\subset (-1,0),\quad \int_\RR \rho(s)ds=1.$$
Then, $g_n$ satisfies
\beq\label{e1bTh2}g_n\in C^\infty(\RR),\quad g_n\hbox{ is non-decreasing},\quad g(s)\leq g_n(s)\leq g\Big(s+{1\over n}\Big)\quad\forall\,s\in\RR.\eeq
Taking $u_n$ the solution of (\ref{ecto0b}) for $g=g_n$, and using estimates (\ref{acowT2}) and (\ref{acotuT2}), we deduce the existence of a subsequence of $n$, still denoted by $n$, and a function $u$ such that
\begin{align}
&u_n\rightharpoonup u\ \hbox{ in }W^{1,p}_0(\Om)\quad\forall\,p<{N\over N-1}\quad(p=2\hbox{ if }N=1),\label{e2Th2}\\
&T_k(u_n)\rightharpoonup T_k(u)\ \hbox{ in }H^1_0(\Om)\quad\forall\,k>0,\label{e3Th2}\\
&\into |g_n(u_n)-g_n(0)|\,dx\le\into|f-g_n(0)|\,dx.\label{e4Th2}
\end{align}
From \eqref{e1Th2} we also have
$$\lim_{n\to\infty}\left(\int_{\{|u_n|<k\}}\hskip-6pt|\nabla u_n|^2\,dx+\into\big(g_n(u_n)-g_n(0)\big)T_k(u_n)\,dx\right)=\into(f-g_+(0))T_k(u)\,dx,$$
for every $k>0$. Dividing by $k$ and taking the limit as $k\to\infty$, gives
\beq\label{e5Th2} \lim_{k\to\infty}\lim_{n\to\infty}\left({1\over k}\int_{\{|u_n|<k\}} \hskip-7pt|\nabla u_n|^2dx+\int_{\{|u_n|>k\}} \hskip-5pt\big|g_n(u_n)|dx\right)=0.\eeq
Let us prove that $g_n(u_n)$ is compact in the weak topology of $L^1(\Om)$. By (\ref{e1bTh2}), (\ref{e4Th2}) and the Dunford-Pettis theorem, it is enough to prove that $g_n(u_n)$ is equi-integrable, i.e. that for every $\ep>0$, there exists $\delta>0$ such that
$$\int_E|g_n(u_n)|dx\le\ep\qquad\forall\,E\subset\Om,\hbox{ measurable, with }|E|<\delta.$$
For such $\ep$, thanks to (\ref{e5Th2}), there exist $k,m>0$ such that 
$$\int_{\{|u_n|>k\}}\hskip-5pt \big|g_n(u_n))|dx<{\ep\over 2}\quad\forall\, n\geq m.$$
Choosing then 
$$\delta_1 <{\ep \over 2\sup_{[-k,k+1]}|g(s)|},$$
and taking into account (\ref{e1bTh2}), we deduce that for every $E\subset\Om$, measurable with $|E|<\delta$, we have
$$\int_E|g_n(u_n)|dx\leq \int_E \,\sup_{[-k,k+1]}|g(s)| dx+\int_{\{|u_n|>k\}}\hskip-8pt |g_n(u_n)|dx<\ep\quad\forall\, n\geq m.$$
On the other hand, since every finite subset of functions in $L^1(\Om)$ is equi-integrable, there exists $\delta_2>0$ such that
$$\int_E|g_n(u_n)|dx< \ep\quad\forall\, n< m.$$
Thus, taking $\delta=\min\{\delta_1,\delta_2\}$ we deduce the equi-integrability of $g_n(u_n)$. Extracting a subsequence if necessary, we then deduce the existence of $w\in L^1(\Om)$ such that
\beq\label{e6Th2} g_n(u_n)\rightharpoonup w\ \hbox{ in }L^1(\Om).\eeq
From (\ref{e1bTh2}), (\ref{e2Th2}) and the Rellich-Kondrachov compactness theorem, we also have
$$g_-(u)\leq w\leq g_+(u)\quad\hbox{ a.e. in }\Om.$$
Let us prove that $(u,w)$ satisfies (\ref{Ecv}). We take $v\in H^1_0(\Om)\cap L^\infty(\Om)$ such that there exists $k>0$ with $\nabla v=0$ a.e. in $\{|u|>k\}.$ For $m>0$, we use $S_m(u_n)v$ as test function in the equation for $u_n$. We get
\[\begin{split}
&\dis -{1\over m}\int_{\{m<|u_n|<2m\}}\hskip-12pt |\nabla u_n|^2\sgn(u_n)v\,dx+\into \nabla u_n\cdot\nabla v\, S_m(u_n)\,dx\\
&\qquad\qquad\dis +\into g_n(u_n)S_m(u_n)v\,dx=\into fS_m(u_n)v\,dx.
\end{split}\]
Taking into account \eqref{e3Th2}, \eqref{e6Th2} and the fact that $S_m(u_n)v$ is bounded in $L^\infty(\Om)$ and converges in measure to $S_m(u)v$, we can pass to the limit as $n\to\infty$ in this inequality to get
$$\ba{l}\dis \left|\into \nabla u\cdot\nabla v\, S_m(u)\,dx+\into wS_m(u)v\,dx-\into fS_m(u)v\,dx\right|\\ \ecart\dis \leq {\|v\|_{L^\infty(\Om)}\over m}\limsup_{n\to \infty}\int_{\{m<|u_n|<2m\}}\hskip-18pt |\nabla u_n|^2dx.\ea$$
By \eqref{e5Th2}, $T_k(u)\in H^1_0(\Om)$ and $\nabla v=0$ a.e. in $\{|u|>k\}$, we can now pass to the limit as $m\to\infty$ to deduce that \eqref{Ecv} holds.\par
The uniqueness of solutions for \eqref{ecto0} follows from Theorem \ref{Thunic}.
\end{proof}

\begin{proof}[Proof of Theorem \ref{Thunic}] Let us assume $N\ge2$. The case $N=1$ is simpler taking into account the continuous imbedding of $L^1(\Om)$ into $H^{-1}(\Om)$.\par
For $m,k>0$, we take $S_m(u_1)S_m(u_2)T_k(u_1-u_2)^+$ as test function in the difference of the equations satisfied by $(u_1,w_1)$ and $(u_2,w_2)$. This gives
\beq\label{e1T3}\ba{l}\dis \int_{\{0<u_1-u_2<k\}}\hskip-12pt |\nabla (u_1-u_2)|^2S_m(u_1)S_m(u_2)dx\\ \ecart\dis
+\into (w_1-w_2)T_k(u_1-u_2)^+S_m(u_1)S_m(u_2)\,dx\\ \ecart\dis
-{1\over m}\int_{\{m<|u_2|<2m\}} \hskip-16ptT_k(u_1-u_2)^+\sgn(u_2)S_m(u_1)\nabla (u_1-u_2)\cdot\nabla u_2\,dx \\ \ecart\dis
-{1\over m}\int_{\{m<|u_1|<2m\}}\hskip-16pt T_k(u_1-u_2)^+\sgn(u_1)S_m(u_2)\nabla (u_1-u_2)\cdot\nabla u_1\,dx\\ \ecart\dis=\into (f_1-f_2)\,T_k(u_1-u_2)^+S_m(u_1)S_m(u_2)\,dx.\ea\eeq
Here we use the second estimate in (\ref{acotuT2}) for $u=u_1$, $u=u_2$, with $k=m$, which dividing by $m$ and taking the limit as $m\to\infty$, gives
$$\lim_{m\to\infty}{1\over m}\int_{\{|u_1 |<2m\}}\hskip-12pt|\nabla u_1|^2dx={1\over m}\lim_{m\to\infty}\int_{\{|u_2 |<2m\}}\hskip-12pt|\nabla u_2|^2dx=0.$$
This allows us to pass to the limit in \eqref{e1T3}, as $m\to\infty$, to deduce
\beq\label{e2T3}\ba{l}\dis \int_{\{0<u_1-u_2<k\}}\hskip-16pt |\nabla (u_1-u_2)|^2dx+\into (w_1-w_2)T_k(u_1-u_2)^+dx\\
\ecart\dis =\into (f_1-f_2)T_k(u_1-u_2)^+dx,\qquad\forall\, k>0.\ea\eeq
Now, we observe that the conditions
$$g_-(u_1)\le w_1\le g_+(u_1),\ \ g_-(u_2)\le w_2\le g_+(u_2)\qquad\hbox{ a.e in }\Om,$$
imply
$$(w_1-w_2)T_k(u_1-u_2)^+\geq 0\qquad\hbox{a.e. in }\Om.$$
Therefore (\ref{e2T3}) proves
$$\int_{\{0<u_1-u_2<k\}}\hskip-16pt |\nabla (u_1-u_2)|^2\,dx\le k\int_{\{0<u_1-u_2<k\}}\hskip-16pt|f_1-f_2|\,dx.$$
Adding the analogous inequality with $u_1,u_2$ replaced by each other, we deduce \eqref{estdu1}. This inequality also implies \eqref{estdu2} (see \cite{BocGal}).

Dividing by $k$ in \eqref{e2T3} and passing to the limit as $k\to0$, we get
$$\lim_{k\to 0}{1\over k}\int_{\{0<u_1-u_2<k\}}\hskip-16pt|\nabla (u_1-u_2)|^2\,dx+\int_{\{u_2<u_1\}}\hskip-12pt|w_1-w_2|\,dx=\int_{\{u_2<u_1\}}\hskip-12pt (f_1-f_2)\,dx.$$
Using the analogous equality with $u_1,u_2$ replaced by each other we conclude \eqref{estdw}.

If $f_1\le f_2$ a.e. in $\Om$, then \eqref{e2T3} proves
$$\int_{\{0<u_1-u_2<k\}}\hskip-16pt|\nabla (u_1-u_2)|^2\,dx=0\qquad\forall\,k>0,$$
and then that $u_1\le u_2$ a.e. in $\Om$.
\end{proof}

\begin{proof} [Proof of Theorem \ref{Threg}] Let us first assume $g$ in $W^{1,\infty}(\RR)$, $f\in W^{1,1}(\Om)\cap L^2(\Om)$. Then $u$ belongs to $H^2(\Om)$. Taking into account the boundary condition $u=0$ on $\partial\Om$, and $\Om\in C^{1,1}$, we can use equation (4.19) in the proof of Lemma 4.3 in \cite{CaCoVa} to prove that the second derivatives of $u$ satisfy
\beq\label{edsg}
\begin{cases}
-\Delta\partial_i u+g'(u)\partial_i u=\partial_i f&\hbox{in }\Om,\quad1\le i\le N\\
\nabla u=|\nabla u|s\nu&\hbox{on }\partial\Om\\
-D^2u\,\nu\cdot\nu+g(0)=f+h\cdot\nabla u&\hbox{on }\partial\Om,
\end{cases}\eeq
where $\nu$ denotes the unitary outside normal to $\Om$, $h,s$ satisfy
$$h\in L^\infty(\partial\Om)^N,\quad s\in L^\infty(\partial\Om),\ s\in \{-1,1\},\hbox{ a.e. in }\Om.$$
and they only depend on $\Om$.

For $\ep>0$, we multiply the first equation in \eqref{edsg} by
$${\partial_i u\over |\nabla u|+\ep}\in H^1(\Om)\cap L^\infty(\Om).$$
Integrating by parts, adding in $i$, and using the boundary conditions in \eqref{edsg}, we get
$$\ba{l}\dis\into\Big({|D^2u|^2\over|\nabla u|+\ep}-{|D^2u\nabla u|^2\over|\nabla u|(|\nabla u|+\ep)^2}\Big)\,dx+\into g'(u){|\nabla u|^2\over |\nabla u|+\ep}\,dx\\
\ecart\dis=-\int_{\partial\Om}s\big(f+h\cdot\nabla u-g(0)\big){|\nabla u|\over|\nabla u|+\ep}\,ds(x)+\into{\nabla f\cdot\nabla u\over|\nabla u|+\ep}\,dx.\ea$$
Using here that
$$|D^2 u|^2-{|D^2 u\nabla u|^2\over|\nabla u|(|\nabla u|+\ep)}\geq |D^2 u|^2-{|D^2 u\nabla u|^2\over|\nabla u|^2}=\Big|D^2 u\Big(I-{\nabla u\otimes \nabla u\over |\nabla u|^2}\Big)\Big|^2\geq 0,$$
a.e. in $\Om$, we can pass to the limit as $\ep\to0$ to deduce
\beq\label{e1Thre}\ba{l}\dis\into{1\over|\nabla u|}\Big(|D^2u|^2-{|D^2u\nabla u|^2\over|\nabla u|^2}\Big)\,dx+\into g'(u)|\nabla u|\,dx\\
\ecart\dis=-\int_{\partial\Om}s\big(f+h\cdot\nabla u-g(0)\big)\,d\sigma(x)+\into{\nabla f\cdot\nabla u\over|\nabla u|}\,dx.\ea\eeq
In the first term of the right-hand side, we can apply the trace theorem for functions in $W^{1,1}(\Om)$ which proves the existence of $C$ depending only on $\Om$ such that
\beq\label{e2Thre}\|f\|_{L^1(\partial\Om)}\leq C\|\nabla f\|_{L^1(\Om)^N}.\eeq
Using also that the imbedding of $W^{1,1}(\Om)$ into $L^p(\Om)$ is compact for $1<p<{N\over N-1}$, we deduce that for every $\delta>0$, there exists $C>0$ depending on $\Om$, $p$ and $\delta$, such that
$$\|\Delta u\|_{L^p(\Om)}\leq \delta \big(\|\nabla g(u)\|_{L^1(\Om)^N}+\|\nabla f\|_{L^1(\Om)^N}\big)+C\big(\|g(u)\|_{L^1\Om)}+\|f\|_{L^1(\Om)}\big).$$
Applying then that $(-\Delta)^{-1}$ is continuous from $L^p(\Om)$ into $W^{1,p}_0(\Om)\cap W^{2,p}(\Om)$, and the trace theorem for Sobolev spaces, we conclude that, for another constant $C>0$, we have
$$\|\nabla u\|_{L^1(\partial\Om)^N}\le\delta \big(\|\nabla g(u)\|_{L^1(\Om)^N}+\|\nabla f\|_{L^1(\Om)^N}\big)+C\big(\|g(u)\|_{L^1\Om)}+\|f\|_{L^1(\Om)}\big),$$
which combined with (\ref{acowT2}), with $w=g(u)$, proves
\beq\label{e3Thre}\|\nabla u\|_{L^1(\partial\Om)^N}\le\delta\|\nabla g(u)\|_{L^1(\Om)^N}+C\big(\|f\|_{W^{1,1}(\Om)}+|g(0)|),\eeq
with $C$ depending on $\Om$ and $\delta$.

Choosing $\delta$ such that $\delta\|h\|_{L^\infty(\partial\Om)}<1$, we can use \eqref{e2Thre} and \eqref{e3Thre} in \eqref{e1Thre}, to conclude the existence of $C>0$ depending only on $\Om$ such that \eqref{acowTBV} holds. From this estimate, the continuous imbedding of $W^{1,1}(\Om)$ into $L^{N\over N-1}(\Om)$, and $u$ solution of (\ref{ecto0}), with $w=g(u)$, we also have that $u$ satisfies \eqref{acouTBV}. This proves the result for $g\in W^{1,\infty}(\RR)$, $f\in W^{1,1}(\Om)\cap L^2(\Om)$.

The case where $g$ is just an increasing function follows by approximating $g$ by a sequence of smooth functions $g_n$ as in the proof of Theorem \ref{Th2ex}.

The case where $f$ is just in $BV(\Om)$ follows by replacing $f$ by a sequence $f_n$ in $W^{1,1}(\Om)\cap L^2(\Om)$, such that $f_n$ converges to $f$ in $L^1(\Om)$ and
$\|\nabla f_n\|_{L^1(\Om)^N}$ converges to $\|\nabla f\|_{{\cal M}(\Om)^N}$.
\end{proof}

\section{Applications to optimal potentials problems}\label{S2}

In this section, we are interested in the study of an optimal control problem for an elliptic equation of a Schr\"odinger type, where the control variable is the potential. Namely, we consider the problem
\beq\label{pbco}\begin{split}
&\dis\min\into\big(j(x,u)+\psi(m)\big)\,dx\\
&\begin{cases}
-\Delta u+m\,u=f\ \hbox{ in }\Om\\
u\in H^1_0(\Om),\quad m\in L^1(\Om),\quad m\ge0\ \hbox{ a.e. in }\Om,
\end{cases}\end{split}\eeq
where $\psi$ is a lower semicontinuous convex function. The constraint $m\ge0$ in $\Om$ is introduced in the cost functional taking
$$\psi(s)=+\infty\qquad\forall\, s<0.$$
Problems of this type intervene in several shape optimization problems, where $m$ is a Borel measure of capacitary type, not Radon in general (see for instance \cite{bebuve}, \cite{bucur}, \cite{bubu05}, \cite{bubuve}, \cite{CaCoVa}, \cite{bdm91}, \cite{bdm93}, \cite{bmv18}, \cite{buve17}, \cite{CFL} and \cite{mazpra}). Similarly to these papers, the solution $u$ of \eqref{pbco} must be understood in the variational sense
$$\begin{cases}
u\in H^1_0(\Om)\cap L^2_m(\Om)\\
\dis\into\big(\nabla u\cdot\nabla v +muv\big)\,dx=\langle f,v\rangle\quad\forall\, v\in H^1_0(\Om)\cap L^2_m(\Om).
\end{cases}$$

In the present paper, we are interested in obtaining some regularity results for the optimal controls $\hat m$, which in our case are integrable functions. In particular, we show that in several cases the optimal controls $\hat m$ are of bang-bang type, and then discontinuous. However, we show that, under some suitable assumptions, $\hat m$ is a $BV$ function and then that the interfaces have finite perimeter.

Our first result proves the existence of solution for \eqref{pbco}. We refer to Theorem 2.19 in \cite{CaCoVa} for a related result in a more general setting.

\begin{Thm}\label{ExCoo} Let $\Om\subset\RR^N$ be a bounded open set, $j:\Om\times\RR\to\RR$ measurable in the first component and lower semicontinuous in the second one, satisfying \eqref{condj}, and $\psi:\RR\to[0,\infty]$ a convex lower semicontinuous function with $\dom(\psi)\subset[0,\infty)$, such that \eqref{psisl} holds. Then, for every $f\in H^{-1}(\Om)$, problem \eqref{pbco} has a least one solution $\hat m\in L^1(\Om)$.
\end{Thm}

The optimality conditions for \eqref{pbco} are given by Theorem \ref{TConOp} below (see also Theorem 4.1 in \cite{CaCoVa}). Since our aim in the present work is to present some regularity conditions for the solutions of problem \eqref{pbco}, let us assume that the right-hand side $f$ in the state equation satisfies
\beq\label{Pref}
f\in W^{-1,r}(\Om),\quad\text{with }
\begin{cases}
r\ge2&\text{if }N=1\\
r>N&\text{if }N\ge2.
\end{cases}\eeq
By Stampacchia's estimates (see for instance \cite{Sta}), this implies that there exists $M>0$, such that for every $m\in L^1(\Om)$, $m\ge 0$ a.e. in $\Om$, the solution $u$ of the state equation in \eqref{pbco} is in $L^\infty(\Om)$, with $\|u\|_{L^\infty(\Om)}\leq M$. In particular, this means that the value of $j(x,s)$ for $|s|\ge M$ is not important and then we can replace in Theorem \ref{ExCoo} condition \eqref{condj} by 
$$\inf_{\{|s|\le M\}}j(x,s)\in L^1(\Om)\qquad\forall\,M>0.$$
Further we will assume $j(x,.)\in C^1(\RR)$ and satisfying
\beq\label{acoFT1} j(.,0)\in L^1(\Om),\qquad \max_{|s|\leq M}|\partial_s j(\cdot,s)|\in L^{r\over 2}(\Om)\quad\forall\, M>0.\eeq
In these conditions, the following result holds.

\begin{Thm}\label{TConOp} Assume that in Theorem \ref{ExCoo}, the right-hand side $f$ and the function $j$ satisfy (\ref{Pref}) and (\ref{acoFT1}) respectively, and define $h:\RR\to\RR$ by
\beq\label{Defh} h(\tau)=\max\big\{s\in \dom(\psi)\ :\ \tau\in\partial\psi(s)\big\}.\eeq
Then, if $\hat m$ is a solution of (\ref{pbco}), $\hat u$ is the corresponding state function, solution of
\beq\label{ostafu}
\begin{cases}
-\Delta\hat u+\hat m\,\hat u=f&\text{in }\Om\\
\hat u=0&\text{on }\partial\Om,
\end{cases}\eeq
and $\hat z$ is the adjoint state, solution of
\beq\label{ostad}
\begin{cases}
-\Delta\hat z+\hat m\,\hat z=\partial_sj(x,\hat u)&\text{in }\Om\\
\hat z=0&\text{on }\partial\Om,
\end{cases}\eeq
we have
\beq\label{condohm} \hat m\in L^\infty(\Om),\quad \hat u\hat z\in \partial\psi(\hat m),\ \ h_-(\hat u\hat z)\leq \hat m\leq h(\hat u\hat z),\quad\hbox{ a.e. in }\Om. \eeq
\end{Thm}

\begin{Rem} \label{Exista} Taking into account that $\dom(\psi)\subset [0,\infty)$, $\psi$ lower semicontinuous and (\ref{psisl}) we deduce that 
$$\lim_{s\to\infty}d_-\psi(s)=\infty,$$
and that, taking 
\beq\label{defal} \alpha:=\inf\dom(\psi)\geq 0,\eeq
one of the following conditions hold
$$\lim_{s\searrow \alpha}\psi(s)=\infty,\ \lim_{s\searrow \alpha}d_+\psi(s)=-\infty\quad \hbox{ or }\quad \alpha\in \dom(\psi),\ d_-\psi(\alpha)=-\infty.$$
Therefore, for every $\tau\in \RR$, there exists $s\in \dom(\psi)$ such that $\tau\in \partial\psi(s)$.
\end{Rem}

\begin{Rem}\label{Rdah} By \eqref{carsd} and \eqref{repsp}, we have that $h$ in Theorem \ref{TConOp}) is also given by
$$h(\tau)=d_+\psi^\ast(\tau)\qquad\forall\,\tau\in\RR.$$
It is always a non-decreasing function, continuous on the right. Moreover, it also satisfies
\beq\label{defh3} h(\tau)=\max\big\{s\in \dom(\psi):\ \tau\geq d_-\psi(s)\big\}.\eeq
\end{Rem}

From \eqref{condohm} and the regularity results for elliptic equations, we deduce that the optimal measure $\hat m$ is very regular if $h$ is and the functions $j$ and $f$ are very regular too. In this sense, the following proposition provides some necessary and sufficient conditions to have $h$ continuous and locally Lipschitz-continuous respectively. 

\begin{Pro}\label{Pregh} The function $h$ defined by (\ref{Defh}) satisfies
\beq\label{resupmu} h\in C^0(\RR)\Longleftrightarrow \psi\ \hbox{ is strictly convex. }\eeq
\beq\label{resLimu} h\in{\rm Lip}(\RR)\iff 0<\inf_{s_1,s_2\in \dom(\psi)\atop s_1<s_2}{d_-\psi(s_2)-d_+\psi(s_1)\over s_2-s_1}.\eeq
\end{Pro}

Further assumptions on $\psi$ than those in Proposition \ref{Pregh} provide more regularity for $h$, but it is interesting to note that 
$$d_+\psi(\alpha)>-\infty\Longrightarrow \partial \psi(\alpha)=\big(-\infty,d_+\psi(\alpha)\big]\Longrightarrow h(s)=\alpha,\ \forall\, s\in (-\infty, d_+\psi(\alpha)\big),$$
with $\alpha$ defined by \eqref{defal}. Thus, $h$ cannot be an analytic function if $d_+\psi(\alpha)>-\infty.$ Even more, we have the following result.

\begin{Pro}\label{deh0}
Assume that $\alpha$ defined by (\ref{defal}) is such that $d_+\psi(\alpha)>-\infty$. Then 
\beq\label{coderal}\exists\, h'(d_+\psi(\alpha))\iff \lim_{s\searrow \alpha }{d_+\psi(s)-d_+\psi(\alpha)\over s-\alpha}=\infty.\eeq
\end{Pro}

From Proposition \ref{Pregh}, we have that $h$ (and then $\hat m$) is not continuous if $\psi$ is not strictly convex. Moreover, Proposition \ref{deh0} shows that even if $h$ is continuous, it is not derivable in general. Theorem \ref{Topr} below provides a sufficient condition to get $\hat m \hat u\hat z$ in $BV(\Om)$, and then shows that the discontinuity surfaces of $\hat m$ have finite perimeter.

\begin{Thm} \label{Topr} In addition to the conditions in Theorem \ref{ExCoo} we assume $\Om\in C^{1,1}$,
\beq\label{defgs} g(\tau):=h(\tau)\tau\qquad\forall\,\tau\in\RR,\eeq
non-decreasing in $\tau$, and
\beq\label{HipFc} \max_{|s|\leq M}|\nabla_x\partial_s j(.,s)|\in L^q(\Om),\quad \max_{|s|\leq M}|\partial^2_{ss} j(.,s)|\in L^1(\Om),\quad \forall\,M>0,\eeq
with 
\beq\label{defp}
q\ge{2N\over N+1}\ \hbox{ if }1\le N\le2,\qquad q>{N\over2}\ \hbox{ if }N\ge3.\eeq
Then, for every $f\in BV(\Om)\cap L^q(\Om)$ and every solution $\hat m$ of \eqref{pbco}, we have
$$\hat u,\hat z\in W^{2,q}(\Om),\qquad\hat m\hat u\hat z\in BV(\Om),$$
with $\hat u$, $\hat z$ the solutions of \eqref{ostafu} and \eqref{ostad} respectively.
\end{Thm}

\begin{Rem} In the assumptions of Theorem \ref{Topr}, the functions $\hat u$, $\hat z$ are continuous, and thus, the set $E:=\{\hat u\hat z=0\}$ is a closed subset of $\overline\Om$ (which contains the boundary). The fact that $\hat m\hat u\hat z$ belongs to $BV(\Om)$, proves then that $\hat m$ belongs to $BV_{loc}(\Om\setminus E)$.
\end{Rem}

\begin{Rem}\label{sufcogc} Since $h$ is non-decreasing, a sufficient condition to have $g$ non-decreasing is to assume $d_+\psi(\alpha)\geq 0$, with $\alpha$ defined by \eqref{defal}.
\end{Rem}

\begin{proof} [Proof of Theorem \ref{ExCoo}] In order to prove the existence of solution, we apply the direct method of the calculus of variations. We take $m_n\in L^1(\Om)$, $m_n\ge0$ a.e. in $\Om$, such that the solution $u_n$ of the state equation in \eqref{pbco} satisfies
$$\lim_{n\to\infty}\into\big(j(x,u_n)+\psi(m_n)\big)dx={\cal I},$$
where we have denoted by ${\cal I}$ the infimum of \eqref{pbco}. In particular
$$\limsup_{n\to\infty}\into\psi(m_n)\,dx<\infty,$$
which, taking into account \eqref{psisl}, implies that $m_n$ is compact in the weak topology of $L^1(\Om)$. Moreover, $f\in H^{-1}(\Om)$ implies that $u_n$ is bounded in $H^1_0(\Om)$. Therefore, extracting a subsequence if necessary, there exist $\hat m\in L^1(\Om)$, and $\hat u\in H^1_0(\Om)$ such that
\beq\label{e1Topr} m_n\rightharpoonup \hat m \ \hbox{ in }L^1(\Om),\qquad u_n\rightharpoonup \hat u \ \hbox{ in }H^1_0(\Om).\eeq
From these convergences, the Rellich-Kondrachov compactness theorem, the lower semicontinuity of $\psi$, and Fatou's Lemma, we deduce
\[\begin{split}
{\cal I}&=\lim_{n\to\infty}\into\big(j(x,u_n)+\psi(m_n)\big)\,dx\\
&=\liminf_{n\to\infty}\into\big(j(x,u_n)-a+c|u_n|^2+\psi(m_n)\big)\,dx-\into\big(-a+c|u|^2)\,dx\\
&\ge\into \big(j(x,\hat u)+\psi(\hat m)\big)\,dx.
\end{split}\]
Proving then that $\hat u$ is the solution of \eqref{ostafu}, we will deduce that $\hat m$ is a solution of \eqref{pbco}. For this purpose, given $v\in H^1_0(\Om)\cap L^\infty(\Om)$, and $l>0$, we take $S_l(u_n)v$ as test function in the equation satisfied by $u_n$. This gives
$$-{1\over l}\into|\nabla u_n|^2v\sgn(u_n)\,dx+\into\nabla u_n\cdot\nabla v\,S_l(u_n)\,dx+\into m_nu_nS_l(u_n)v\,dx=\langle f,S_l(u_n)v\rangle.$$
By \eqref{e1Topr}, the Rellich-Kondrachov compactness theorem, and $S_l(u_n)v$ bounded in $L^\infty(\Om)$, we can pass to the limit in $n$, in the three last terms in this equality. In the first term, we can use that $u_n$ is bounded in $H^1_0(\Om)$ and that $v$ belongs to $L^\infty(\Om)$. Thus, there exists $C>0$ independent of $m$ such that
$$\left|\into\big(\nabla\hat u\cdot\nabla v+\hat m\hat uv)S_l(\hat u)\,dx-\langle f,S_l(\hat u)v\rangle\right|\le{C\over l}.$$
Using that $S_l(\hat u)v$ converges strongly to $v$ in $H^1_0(\Om)$ as $l\to\infty$, and the Lebesgue dominated convergence theorem, we can pass to the limit as $l\to\infty$ in this equality to get
$$\into \big(\nabla\hat u\cdot\nabla v+\hat m\hat uv)\,dx=\langle f,v\rangle,\quad\forall\,v\in H^1_0(\Om)\cap L^\infty(\Om).$$
If $v$ is just in $H^1_0(\Om)\cap L^2_m(\Om)$, we prove that this equality also holds true just replacing $v$ by $T_k(v)$ and then passing to the limit as $k\to\infty$.
\end{proof}

\begin{proof} [Proof of Theorem \ref{TConOp}]
Let $\hat m$ be a solution of \eqref{pbco} and define $\hat u$, $\hat z$ as the solutions of (\ref{ostafu}) and \eqref{ostad} respectively. Since $f$ satisfies (\ref{Pref}), Stampacchia's estimates (see \cite{Sta}) show that $\hat u$ belongs to $H^1_0(\Om)\cap L^\infty(\Om)$. By \eqref{acoFT1} we then have that $\hat z$ solution of (\ref{ostad}) is well defined and belongs to $H^1_0(\Om)\cap L^2_{\hat m}(\Om)$.

Now, for another non-negative function $m\in L^1(\Om)$ such that $\psi(m)$ is integrable, $k>0$, and $\ep\in (0,1]$, we define $m_\ep$ as
$$m_\ep=\hat m +\ep\, T_k(m -\hat m),$$
and $u_\ep$ as the solution of 
$$\begin{cases}
-\Delta u_\ep+m_\ep u_\ep=f&\hbox{in }\Om\\
u_\ep=0&\hbox{on }\partial\Om.\end{cases}$$
Taking into account that 
$$m_\ep=\begin{cases}
(1-\ep)\hat m+\ep m&\hbox{if }|m-\hat m|\le k\\
\dis\Big(1-{\ep k\over|m-\hat m|}\Big)\hat m+{\ep k\over|m-\hat m|}m&\hbox{if }|m-\hat m|>k,
\end{cases}$$
that $\hat m$ is a solution of \eqref{pbco}, and the convexity of $\psi$, we deduce
\[\begin{split}
&\into j(x,u_\ep)+\int_{\{|m-\hat m|\le k\}}\hskip-8pt\big((1-\ep)\psi(\hat m)+\ep\,\psi(m)\big)\,dx\\
&\qquad+\int_{\{|m-\hat m|> k\}}\hskip-8pt\Big(\Big(1-{\ep k\over |m-\hat m|}\Big)\psi(\hat m)+{\ep k\over |m-\hat m|} \psi(m)\Big)\,dx\\
&\qquad\ge\into \big(j(x,u_\ep)+ \psi(m_\ep)\big)\,dx
\ge\into \big(j(x,\hat u)+ \psi(\hat m)\big)dx,
\end{split}\]
for every $\ep>0$. Using then that
$${u_\ep-u\over\ep}\rightarrow u'\ \hbox { in }H^1_0(\Om),$$
with $u'\in H^1_0(\Om)\cap L^\infty(\Om)$ the solution of
$$\begin{cases}
-\Delta u'+\hat m u'+T_k(m-\hat m)\hat u=0&\hbox{in }\Om\\
u'=0&\hbox{on }\partial\Om,\end{cases}$$
we get
\beq\label{e3Topr}
\begin{split}\into\partial_s j(x,\hat u)u'dx+\int_{\{|m-\hat m|\le k\}}\hskip-10pt\big(\psi(m)-\psi(\hat m)\big)\,dx\\
\qquad+\int_{\{|m-\hat m|>k\}}\hskip-10pt{k\big(\psi(m)-\psi(\hat m)\big)\over |m-\hat m|}\,dx\ge0.
\end{split}\eeq
On the other hand, taking $u'$ as test function in \eqref{ostad}, and $\hat z$ as test function in \eqref{ostafu} we deduce
$$\into\partial_sj(x,u)u'\,dx=\into\big(\nabla\hat z\cdot\nabla u'+\hat m\hat z u'\big)\,dx=-\into T_k(m-\hat m)\hat u\hat z\,dx.$$
Therefore \eqref{e3Topr} provides
$$-\into T_k(m-\hat m)\hat u\hat z\,dx+\int_{\{|m-\hat m|\leq k\}} \hskip-10pt\big(\psi(m)-\psi(\hat m)\big)\,dx+\int_{\{|m-\hat m|> k\}} \hskip-10pt { k\big(\psi(m)-\psi(\hat m)\big)\over |m-\hat m|}\,dx\ge0.$$
Using here that $m$, $\hat m$, $\psi(m)$ and $\psi(\hat m)$ belong to $L^1(\Om)$ and that the second assertion in \eqref{acoFT1}, $\hat z$ solution of \eqref{ostad} and Stampacchia's estimates imply that $\hat z$ is in $L^\infty(\Om)$, we can take the limit as $k\to\infty$ to deduce
\[\begin{split}
&\into\big(\psi(m)-m\hat u\hat z\big)\,dx\ge\into\big(\psi(\hat m)-\hat m\hat u\hat z\big)\,dx,\\
&\forall\,m\in L^1(\Om),\ m\ge0\hbox{ a.e. in }\Om,\ \into\psi(m)\,dx<\infty.
\end{split}\]
This implies that $\hat m$ satisfies
$$\hat m\in\dom(\psi),\quad\psi(\hat m)-\hat m\hat u\hat z=\min_{s\in \dom(\psi)}\big\{\psi(s)-s\hat u\hat z\big\},\quad\hbox{a.e. in }\Om,$$
or equivalently, 
$$\hat m\hat u\hat z=\psi(\hat m)+\psi^\ast(\hat u\hat z).$$
By \eqref{repsp} this is also equivalent to $\hat m\in \partial\psi^\ast(\hat u\hat z)$, and also to $\hat u\hat z\in \partial\psi(\hat m).$ From (\ref{carsd}) applied to $\psi^\ast$ and Remark \ref{Rdah} we then deduce the third assertion in (\ref{condohm}). Combined with $\hat u$ and $\hat z$ in $L^\infty(\Om)$, this also implies that $\hat m$ is in $L^\infty(\Om)$.
\end{proof}

\begin{proof} [Proof of Proposition \ref{Pregh}]
Let us prove \eqref{resupmu}. If $\psi$ is strictly convex, the quotient function defined by (\ref{apico}) restricted to the set
$$\big\{(s,t)\in\dom(\psi)\times\dom(\psi)\ :\ s<t\big\},$$
is strictly increasing in $s$ and $t$. This proves
$$d_+\psi(s_1)<d_-\psi(s_2)\qquad\forall\,s_1,s_2\in\dom(\psi),\ s_1<s_2.$$
By Remark \ref{Exista}, this implies that for every $\tau\in\RR$, there exists a unique $s\in\dom(\psi)$ such that $\tau\in\partial\psi(s)$. By definition (\ref{Defh}) of $h$ we deduce that $h(\tau)$ agrees with such $s$. Now, we observe that the lower semicontinuity of $\psi$ and definition \eqref{defsd} of $\partial\psi$ imply the following continuity property for $\partial\psi$:
$$s_n,s\in \dom(\psi),\ \tau_n\in\partial\psi(s_n),\ \tau\in\RR,\quad s_n\to s,\ \tau_n\to\tau\Longrightarrow\tau\in\partial\psi(s).$$
Thanks to the uniqueness of $s$ proved above, this can also be read as
$$s_n=h(\tau_n),\ s_n\to s,\ \tau_n\to\tau\Longrightarrow s=h(\tau),$$
and then proves the continuity of $h$.

For the reciprocal, we argue by contradiction. If $\psi$ is not strictly convex, then, there exists an interval $[c,d]\subset\dom(\psi)$, with $c<d$ such that $\psi$ is an affine function with a certain slope $\la\in\RR$ in this interval. Moreover $c\ge0$ can be chosen as
$$c=\min_{s\in\dom(\psi)} d_+\psi(s)=\la.$$
If $c=\alpha$, defined by \eqref{defal}, then $\partial\psi(c)=(-\infty,\la]$ and thus, definition \eqref{Defh} of $h$ implies
$$h(\tau)=c\qquad\forall\,\tau<\la,\ h(\la)\ge d.$$
If $c>\alpha$, then $d_+\psi(s)<\la$ for every $s<c$, and thus
$$h(\tau)< c\qquad\forall\,\tau<\la,\ h(\la)\ge d.$$
In both cases, we conclude that $h$ is not continuous at $\tau=\lambda$.

In order to prove \eqref{resLimu} we first observe that the right-hand side implies $\psi$ strictly convex. Since $h\in{\rm Lip}_{\rm loc}(\RR)$ implies $h$ continuous, we conclude that the left-hand side in \eqref{resLimu} also implies $\psi$ strictly convex. Therefore it is enough to prove the result for $\psi$ strictly convex. As we saw at the beginning of the proof, this implies that for every $\tau\in \RR$ there exists a unique $s\in\dom(\psi)$ such that $\tau\in\partial\psi(s)$, and this $s$ satisfies $h(\tau)=s.$ Then, taking into account \eqref{carsd}, we deduce the existence of $L>0$ such that
$$h(\tau_2)-h(\tau_1)\leq L(\tau_2-\tau_1),$$
for every $\tau_1\leq\tau_2$, is equivalent to
$$s_2-s_1\leq L\big(d_-\psi(s_1)-d_+\psi(s_2)\big),$$
for every $s_1,s_2\in \dom(\psi)$, and then, that (\ref{resLimu}) holds.
\end{proof}

\begin{proof} [Proof of Proposition \ref{deh0}] Since $d_+\psi(\alpha)>-\infty$, we have $\partial\psi(\alpha)=(-\infty,d_+\psi(\alpha)]$, and then $h(s)=\alpha$ for every $s<d_+\psi(\alpha)$. Therefore, if $h$ is derivable at $d_+(\psi(\alpha))$, we must have
$$h(d_+\psi(\alpha))=\alpha,\qquad\lim_{\ep\searrow 0}{h(d_+\psi(\alpha)+\ep)-\alpha\over\ep}=0.$$
Thus, for every $\rho>0$, there exists $\delta>0$ such that $0<\ep<\delta$ implies
$$h(d_+\psi(\alpha)+\ep)<\alpha+\ep\rho,$$
which by \eqref{defh3} can also be read as
$$s\in\dom (\psi),\ d_-\psi(s)\le d_+\psi(\alpha)+\ep\Longrightarrow s<\alpha+\ep\rho,$$
and then that 
$$s\ge\alpha+\ep\rho\Rightarrow d_-\psi(s)>d_+\psi(\alpha)+\ep.$$
Taking $s=\alpha+\ep \rho$ and letting $\ep\to0$, this gives
$$\lim_{\ep\searrow0}{d_-\psi(\alpha+\ep\rho)-d_+\psi(\alpha)\over\ep\rho}>{1\over\rho}\qquad\forall\,\rho>0,$$
and then 
$$\lim_{\ep\searrow0}{d_-\psi(\alpha+\ep)-d_+\psi(\alpha)\over\ep}=\infty.$$
The proof of the reciprocal follows by a similar argument.
\end{proof}

\begin{proof} [Proof of Theorem \ref{Topr}] Taking into account that $\hat m$ belongs to $L^\infty(\Om)$, that $\Om$ is $C^{1,1}$, $f\in L^q(\Om)$, the first assertion in \eqref{HipFc} and $\hat u$, $\hat z$ solutions of \eqref{ostafu} and \eqref{ostad} respectively, we can apply the regularity results for elliptic equations (see e.g. \cite{GiTr}, chapter 7) to deduce that $\hat u$, $\hat z$ belong to $W^{2,q}(\Om)$.

Now, we use that $\hat u\hat z$ satisfies
$$-\Delta(\hat u\hat z)+\hat m\hat u\hat z=-2\nabla\hat u\cdot\nabla\hat z+f\hat z+\partial_sj(x,\hat u)\hat u\qquad\hbox{in }\Om.$$
where thanks to \eqref{HipFc}, $\hat f\in BV(\Om)\cap L^q(\Om)$ and $\hat u,\hat z\in W^{2,q}(\Om)$, we have that the right-hand side of this equation belongs to $BV(\Om)$. Using then that (\ref{condohm}) and definition (\ref{defgs}) of $g$ imply
$$g_-(\hat u\hat z)\le\hat m\hat u\hat z\le g_-(\hat u\hat z)\qquad\hbox{a.e. in }\Om,$$
and that we are assuming $g$ non-decreasing in $\RR$, we can apply Theorem \ref{Threg} to deduce that $\hat m\hat u\hat z$ belongs to $BV(\Om)$.
\end{proof}

\section{Some examples} \label{S5}
In the present section we illustrate the results obtained in the previous one, by applying them to some classical examples. In Section \ref{S6} we will also perform some numerical computations relative to these examples.

\medskip\noindent
{\bf First example.} We consider the case where we are looking for a non-negative function $m$ such that the solution $u$ of the state equation in (\ref{pbco}) minimizes
\beq\label{fucsm} J(u):=\into j(x,u)\,dx,\eeq
and it is such that the norm of $m$ in some space $L^p(\Om)$, $1<p<\infty$ is not too large. This can be modeled by \eqref{pbco}, with $\psi$ given by
$$\psi(s)=\infty 1_{(-\infty,0)}+ks^p1_{[0,\infty)}.$$
with $k$ a positive to parameter to choose. Then,
$$\partial\psi(s)=\begin{cases}
(-\infty,0]&\hbox{if }s= 0,\\
\{kp\,s^{p-1}\}&\hbox{if }s>0,\end{cases}\qquad h(\tau)=\Big({\tau\over kp}\Big)^{1\over p-1}1_{(0,\infty)}.$$
By Theorem \ref{ExCoo} and by the fact that $\psi$ is strictly convex, we have that $h$ is continuous. It is in $C^1(\RR)$ when $p<2$, i.e. when condition \eqref{coderal} holds, and it is locally Lipschitz continuous if $p\le2$, (and then condition \eqref{resLimu} holds in bounded subsets of $\dom(\psi)$).
 
From \eqref{condohm}, we deduce that, taking $j$ and $f$ regular enough, every solution $\hat m$ of \eqref{pbco} satisfies
\beq\label{ExhmE1}
\hat m=\Big({\hat u\hat z\over kp}\Big)^{1\over p-1}1_{\{\hat u\hat z>0\}}\qquad\hbox{a.e. in }\Om,\eeq
with $\hat u$ and $\hat z$ the solutions of \eqref{ostafu} and \eqref{ostad} respectively. 
In particular, this means that $\hat u$, $\hat z$ solve the nonlinear system
\beq\label{nolsie1}
\begin{cases}
\dis-\Delta\hat u+\Big({\hat u\hat z\over kp}\Big)^{1\over p-1}1_{\{\hat u\hat z>0\}}\hat u=f&\hbox{in }\Om,\\
\dis-\Delta\hat z+\Big({\hat u\hat z\over kp}\Big)^{1\over p-1}1_{\{\hat u\hat z>0\}}\hat z=\partial_sj(x,\hat u)&\hbox{in }\Om,\\
\dis\hat u=\hat z=0&\hbox{on }\partial\Om.
\end{cases}\eeq
By Theorem \ref{Topr} we obtain that
$$\hat m\hat u\hat z={\big(\hat u\hat z\big)^{p'}\over(kp)^{1\over p-1}}\,1_{\{\hat u\hat z>0\}}$$
belongs to $BV(\Om)$ (and even to $W^{1,1}(\Om)$). However, in order to have $\hat m$ in $W^{1,1}(\Om)$, we need $(\hat u\hat z)^{2-p\over p-1}\nabla (\hat u\hat z)1_{\{\hat u\hat z\}}$ to be in $L^1(\Om)^N$, which is not clear for $p>2$, i.e. when $h$ is not locally Lipschitz.

\medskip\noindent
{\bf Second example.} We now consider the case where we want to minimize functional \eqref{fucsm}, with $u$ the solution of the state equation in (\ref{pbco}), under the constraint $m\in[\alpha,\beta]$, with $0\leq\alpha<\beta$. The problem corresponds to \eqref{pbco}, with $\psi$ given by
\beq\label{Psi2ej}\psi(s)=\infty1_{(-\infty,\alpha)\cup (\beta,\infty)}.\eeq
Thus, $\dom(\psi)=[\alpha,\beta]$, and 
$$\partial\psi(s)=\begin{cases}
(-\infty,0]&\hbox{if }s=\alpha\\
\{0\}&\hbox{if }\alpha<s<\beta\\
[0,\infty)&\hbox{if }s=\beta,\end{cases}
\qquad
h(\tau)=\begin{cases}
\alpha&\hbox{if }\tau<0\\
\beta&\hbox{if }\tau\ge0,\end{cases}$$
As expected, since $\psi$ is not strictly convex, we get by \eqref{resupmu} that $h$ is not continuous. Condition \eqref{condohm} reads in this case as
\beq\label{condmu2e}
\begin{cases}
\hat m=\alpha&\hbox{a.e. in }\big\{\hat u\hat z<0\}\\
\hat m=\beta&\hbox{a.e. in }\big\{\hat u\hat z>0\}\\
\hat m\in[\alpha,\beta]&\hbox{a.e. in }\big\{\hat u\hat z=0\}.
\end{cases}\eeq
Therefore, the value of $\hat m$ is not determined on the set where $\hat u\hat z$ vanishes. When this set has zero measure, assertion (\ref{condmu2e}) shows that $\hat m$ is a bang-bang control, which only takes the values $\alpha$ and $\beta$. Taking into account that $g(s)=h(s)s$ is non-decreasing, we deduce from Theorem \ref{Topr} that $\hat m\hat u\hat z$ is in $BV(\Om)$.

As a simple case where we can assure that $\hat m$ is a bang-bang control, we take
\beq\label{fF2ej}
j(x,s)=\gamma(x)s\qquad\hbox{a.e. }x\in\Om,\ \forall\,s\in\RR,\eeq
Assuming $\gamma,f\in L^q(\Om)$, with $q$ satisfying \eqref{defp}, we have that $\hat u,\hat z$ are in $W^{2,q}(\Om)$. Therefore, \eqref{ostafu} and \eqref{ostad} imply
$$f=0\ \hbox{ a.e. in }\{\hat u=0\},\qquad\gamma=0\ \hbox{ a.e. in }\{\hat z=0\}.$$
Thus, assuming $|\{f=0\}|=|\{\gamma=0\}|=0,$ we conclude that the set $\{\hat u\hat z=0\}$ has zero measure.

It is also simple to give a counterexample where the set $\{\hat u\hat z=0\}$ has positive measure and the control $\hat m$ is not a bang-bang control. Just take $\tilde m\in C^0(\overline\Om;[\alpha,\beta])$, not constant, and $\tilde u$ the solution of \eqref{ostafu}, with $\hat m$ replaced by $\tilde m$. Defining
$$j(x,s)=|s-\tilde u(x)|^2\qquad\hbox{a.e. }x\in\Om,\ \forall\,s\in\RR,$$
we deduce that problem \eqref{pbco} has the unique solution $\hat m=\tilde m$, for which the functional vanishes. Observe that in this case $\partial_sj(x,\hat u)=0$ a.e. in $\Om$, and then $\hat z$ is the null function. Therefore the set $\{\hat u\hat z=0\}$ is the whole set $\Om$, and condition \eqref{condmu2e} does not provide any information about $\hat m$.

\medskip\noindent
{\bf Third example.} We consider a mixture of the first and second examples. Now, the goal is to minimize \eqref{fucsm}, with $u$ the solution of the state equation in (\ref{pbco}) and $m\in [\alpha,\beta]$, such that its norm in $L^p(\Om)$ is not too great, with $1\leq p<\infty$. The problem corresponds to take in \eqref{pbco}
$$\psi(s)=\infty1_{(-\infty,\alpha)\cup (\beta,\infty)}+k s^p1_{[\alpha,\beta]},$$
with $k>0$ a positive constant to choose.

In the strictly convex case $p>1$, we have
$$\partial\psi(s)\hskip-2pt=\hskip-2pt\begin{cases}
(-\infty,kp\alpha^{p-1}]&\hbox{if }s=\alpha\\
\{kps^{p-1}\}&\hbox{if }\alpha<s<\beta\\
[kp\beta^{p-1},\infty)&\hbox{if }s=\beta,\end{cases}
\qquad
h(\tau)\hskip-2pt=\hskip-2pt\begin{cases}
\alpha&\hbox{if }\tau<kp\alpha^{p-1}\\
\dis\Big({\tau\over kp}\Big)^{1\over p-1}&\hbox{if }kp\alpha^{p-1}\leq \tau<kp\beta^{p-1}\\
\beta&\hbox{if }\tau\ge kp\beta^{p-1}.\end{cases}$$
As in the first example, the strict convexity of $\psi$ provides a function $h$ which is continuous. Therefore, the optimal controls are continuous and even, they are in some Sobolev space if $p\le2$ (assuming $j$ and $f$ smooth enough).

In the case $p=1$, we have
$$\partial\psi(s)=\begin{cases}
(-\infty,k]&\hbox{if }s=\alpha\\
\{k\}&\hbox{if }\alpha<s<\beta\\
[k,\infty)&\hbox{if }s=\beta,\end{cases}
\qquad
h(\tau)=\begin{cases}
\alpha&\hbox{if }\tau<k\\
\beta&\hbox{if }\tau\ge k.\end{cases}$$
As in the second example, \eqref{condohm} provides
$$\begin{cases}
\hat m=\alpha&\hbox{a.e. in }\big\{\hat u\hat z<k\}\\
\hat m=\beta&\hbox{a.e. in }\big\{\hat u\hat z>k\}\\
\hat m\in[\alpha,\beta]&\hbox{a.e. in }\big\{\hat u\hat z=k\},
\end{cases}$$
and then the optimal controls are bang-bang controls if the set $\{\hat u\hat z=k\}$ has null measure.

Since $h(\tau)\tau$ is still a non-decreasing, Theorem \ref{Topr} proves that $\hat m\hat u\hat z$ is in $BV(\Om)$ for every optimal control $\hat m$. Using also that $\hat m=\alpha$ in a neighborhood of the closed set $\{\hat u\hat z=0\}$ we deduce that in this case $\hat m$ is also in $BV(\Om)$.

As a particular case we can take 
\beq\label{eleme} j(x,s)=f(x)s.\eeq
This is a classical problem corresponding to the minimization of the energy. In this case control problem (\ref{pbco}) can also be written in the simplest form
\beq\label{mien} \max_{\alpha\leq m\leq \beta}\min_{u\in H^1_0(\Om)}\left\{\into \Big(|\nabla u|^2+mu^2-2fu\Big)dx-k\into m\,dx\right\}.\eeq
For $k=0$, it is clear that the solution corresponds to $\hat m=\beta$. Thus, the interesting case corresponds (as we assumed above) to $k>0$. This means that we just want to spend a limited amount of the optimal potential $\beta$ (for example, because it is more expensive).

From \eqref{eleme} we get $\hat z=\hat u$, and then (\ref{condohm}) gives
\beq\label{comuen}\begin{cases}
\hat m=\alpha&\hbox{a.e. in }\{|\hat u|^2<k\}\\
\hat m=\beta&\hbox{a.e. in }\{|\hat u|^2>k\}.
\end{cases}\eeq
Taking $f\in L^q(\Om)$, with $q$ satisfying \eqref{defp}, we have $\hat u$ in $W^{2,q}(\Om)$, and then \eqref{ostafu} provides
$$\sqrt{k}\hat m=f\ \hbox{ a.e. in }\big\{\hat u=\sqrt{k}\big\},\qquad -\sqrt{k}\hat m=f\ \hbox{ a.e. in }\big\{\hat u=-\sqrt{k}\big\}. $$
Therefore, a sufficient condition to have $\hat m$ a bang-bang control is to assume that the set $\{f\in  [-\sqrt{k}\beta,-\sqrt{k}\alpha]\cup [\sqrt{k}\alpha,\sqrt{k}\beta] \}$ has null measure. This holds in particular if $\alpha$ is positive, $f$ belongs to $L^\infty(\Om)$ and $k$ is large enough.

\medskip\noindent
{\bf Fourth example.} Related to the case $p=1$ in the third example, let us take
$$\psi(s)=\infty 1_{(-\infty,\alpha)\cup (\beta,\infty)}-k s1_{[\alpha,\beta]},$$
with $0\leq \alpha<\beta$, $k>0$. In this case we are interested in controls $m$ which take their values in $[\alpha,\beta]$ and its integral is large. Similarly to the third example, we have
$$\partial\psi(s)=\begin{cases}
(-\infty,-k]&\hbox{if }s=\alpha\\
\{-k\}&\hbox{if }\alpha<s<\beta\\
[-k,\infty)&\hbox{if }s=\beta,\end{cases}
\qquad
h(\tau)=\begin{cases}
\alpha&\hbox{if }\tau<-k\\
\beta&\hbox{if }\tau\ge-k,\end{cases}$$
Theorem \ref{ExCoo} proves that the optimal controls satisfy
$$\begin{cases}
\hat m=\alpha&\hbox{a.e. in }\big\{\hat u\hat z<-k\}\\
\hat m=\beta&\hbox{a.e. in }\big\{\hat u\hat z>-k\}\\
\hat m\in[\alpha,\beta]&\hbox{a.e. in }\big\{\hat u\hat z=-k\},
\end{cases}$$
and then they are not continuous in general. Moreover, in this case the function
$$g(\tau)=h(\tau)\tau=\begin{cases}
\alpha\tau&\hbox{if }\tau<-k\\
\beta\tau&\hbox{if }\tau\ge-k,\end{cases}$$
decreases at $\tau=-k$. Thus, this is an example where Theorem \ref{Topr} does not apply and therefore, we do not know if $\hat m\hat u\hat z$ is in $BV(\Om)$.\par
A classical example corresponds to the maximization of the energy i.e. (compare with (\ref{eleme}))
\beq\label{elecom} F(x,s)=-f(x)s.\eeq
Similarly to (\ref{mien}), it is known that the problem can also be written as
\beq\label{maen} \min_{\alpha\leq m\leq \beta}\min_{u\in H^1_0(\Om)}\left\{\into \Big(|\nabla u|^2+mu^2-2fu\Big)dx-k\into m\,dx\right\}.\eeq
For $k=0$ the solution is given by $\hat m=\alpha$ and then the interesting case is $k>0$, as assumed above. Now, the function $\hat z$ is equal to $-\hat u$, and therefore $\hat m$ satisfies (compare with \eqref{comuen}
$$\begin{cases}
\hat m=\alpha&\hbox{a.e. in }\big\{k<|\hat u|^2\}\\
\hat m=\beta&\hbox{a.e. in }\big\{|\hat u|^2<k\}.
\end{cases}$$
Thus, a suffcient condition to assure that the optimal controls only take the values $\alpha$ and $\beta$ is also to asume that the set $\{f\in  [-\sqrt{k}\beta,-\sqrt{k}\alpha]\cup [\sqrt{k}\alpha,\sqrt{k}\beta] \}$ has null measure.

\section{Some numerical simulations}\label{S6}

In this section, we illustrate the results of the previous ones through the numerical resolution, in the 2D case, of problem \eqref{pbco} for the first three examples in Section \ref{S5}. For the first example we will consider $p=2$ and for the third one $p=1$. We apply a gradient descent method with projection. It depends of the function $\psi$ associated to the volume constraint of the potential. The corresponding algorithm is related to Theorem \ref{TConOp} providing the optimality conditions to \eqref{pbco}. We refer to \cite{All}, \cite{Cas} for similar algorithms in optimal design problems. It reads as follows.

\begin{itemize}
\item Initialization: choose an admisible function $m_0\in L^1(\Om)$, such that $\Psi(m_0)<\infty$.
\item For $j\ge0$, iterate until stop condition as follows.
\begin{itemize}
\item Compute $u_j,z_j$ solution of (\ref{ostafu}), (\ref{ostad}), for $\hat m=m_j$.
\item Compute $\tilde m_j$ descent direction associated to $u_j$ and $z_j$, as:
$$\tilde m_j=\begin{cases}
\dis{u_jz_j-2km_j\over\|u_jz_j-2km_j\|_{L^{2}(\Om)}}&\hbox{in the first example,}\\\sgn(u_jz_j)&\hbox{in the second example,}\\
\sgn(k-u_jz_j)&\hbox{in the third example.}
\end{cases}$$
\item Update the function $m_j$:
$$m_{j+1}=P_\psi(m_j+\epsilon_j \tilde m_j)$$
where $P_\psi$ is the projection operator from $\RR$ into the domain of $\psi$, i.e.
$$P_\psi(s)=\begin{cases}
s^+&\hbox{in the first example,}\\
\min\{\beta,\max\{s,\alpha\}\}&\hbox{in the second and third examples.}
\end{cases}$$
\end{itemize}
\item Stop if $\frac{|I(m_j)-I(m_{j-1)}|}{|I(m_0)|}<tol$, for $tol>0$ small. 
\end{itemize}

In all the simulations we have chosen $\Om$ as the ball $B$ of center zero and radius one in dimension two. For the first example, we have chosen a simple case where the solution is radial. Thus, we just solve the corresponding one-dimensional problem. The implementation for this example has been carried out using Matlab R2022a. The second and third examples have been implemented using the free software FreeFemm++ v 4.4-3 (see \cite{Hecht} and http://www.freefem.org/). The corresponding results are as follows.

\medskip\noindent
\textbf{First example.} For $s_0\in\RR$, we take:
$$j(x,s)=\frac12 \big|s-s_0\big|^2,\qquad\psi(s)=\infty 1_{(-\infty,0)}+k|s|^2,\qquad f=1.$$
Taking $m$ as the null function, the solution of the state equation in (\ref{pbco}) is 
$$u(x)={1-|x|^2\over 4}.$$
Thus, the interesting case corresponds to $s_0\in (0,1/4)$. When $k=0$ assumption \eqref{psisl} is not satisfied, and the optimal control is not necessarily given by a function $\hat m\in L^1(\Om)$ (we refer to \cite{bubu05}, \cite{BuCaMa}, \cite{bdm91}, \cite{bdm93}, \cite{bgrv14}, \cite{bmv18}, \cite{buve17} for some other existence results on optimal potentials). Indeed, it can be proved that the optimal control $\hat m$ is given by the Radon measure
$$\hat m={1\over s_0}1_{\{|x|<a\}}dx+{4s_0-1+a^2(1-2\log a)\over 4s_0\log a}1_{\{|x|=a\}}d\sigma,$$
where $d\sigma$ is the 1-dimensional measure and $a$ is characterized by
$$0<a<1,\quad 4s_0-1+a^2<2a^2\log a,$$
$$\log a\int_a^1r(4s_0-1+r^2)\log r\,dr=(4s_0-1+a^2\Big)\int_a^1r\log^2r\,dr.$$
The corresponding optimal control $\hat u$ is given by
$$\hat u(x)=\begin{cases}
s_0 &\hbox{if }|x|<a\\
\dis{1-|x|^2\over4}+{(4s_0-1+a^2)\log|x|\over4\log a}&\hbox{if }a\le|x|\le1.
\end{cases}$$

In Figure \ref{Fig:Ex1CE} we represent the optimal state function $\hat u$ for $s_0=0.1$ (then $a\sim 0.2825$)\par
In Figure \ref{Fig:Ex11} we represent on the top the optimal control and optimal state function obtained by solving the above problem (corresponding to $k=0$) numerically. In the middle and the bottom we also represent the optimal control and optimal state function, taking $k=10^{-7}$ and $k=10^{-5}$ respectively. Since the solutions are radial functions, we have applied the algorithm to the corresponding one-dimensional problem. Observe that in this radial representation the singular part of the optimal measure $\hat m$ is given by two Dirac masses at $r=-a$ and $r=a$. In the numerical computation this provides the two corners which appear in the figure.
\begin{figure}[!h]
\centering
\includegraphics[scale=0.21]{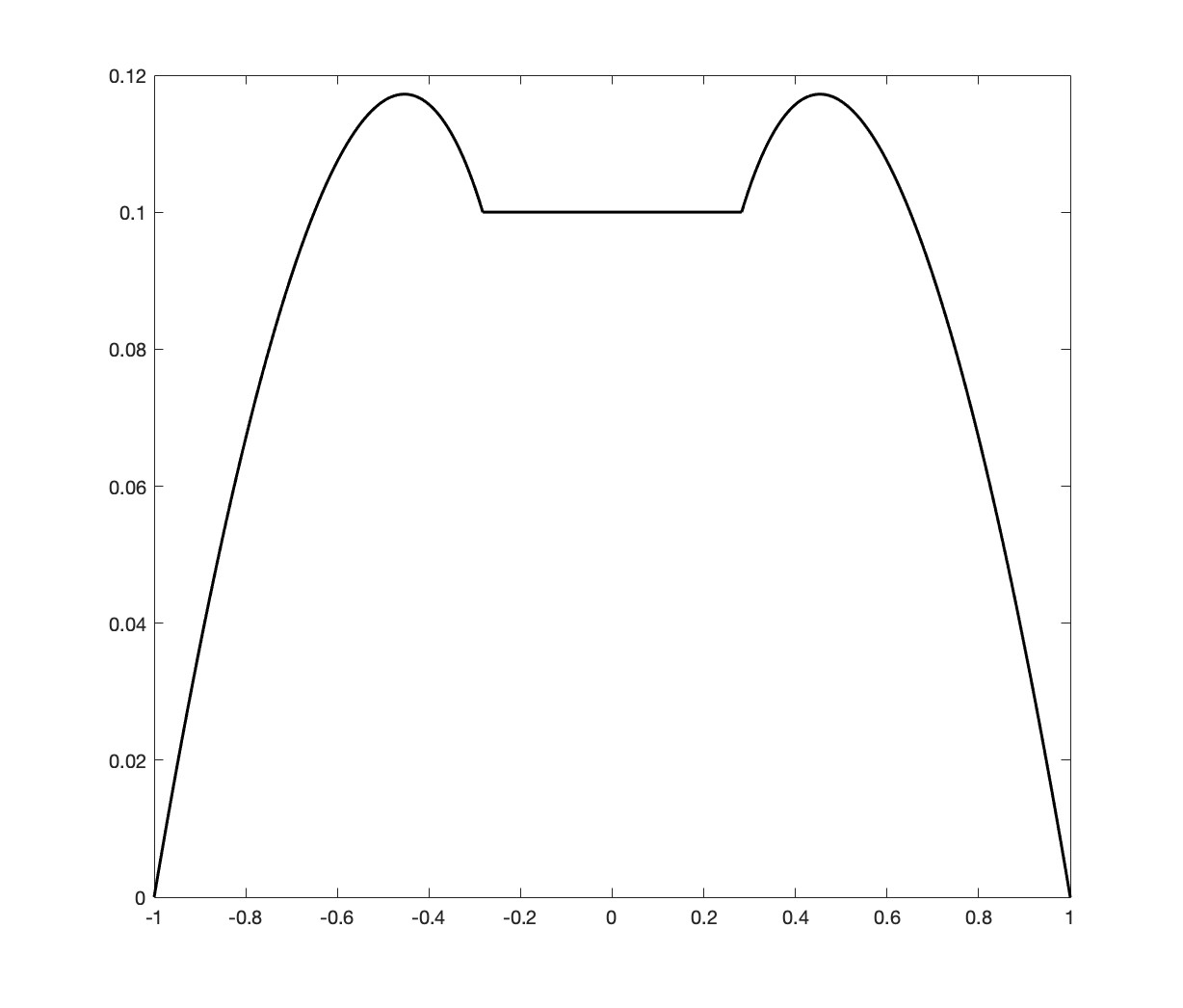}
\caption{Example 1: Optimal state $\hat u$ for $k=0$ (analytic solution).}
\label{Fig:Ex1CE}
\end{figure}

\begin{figure}[!h]
\begin{minipage}[t]{7.4cm}
\centering
\includegraphics[scale=0.21]{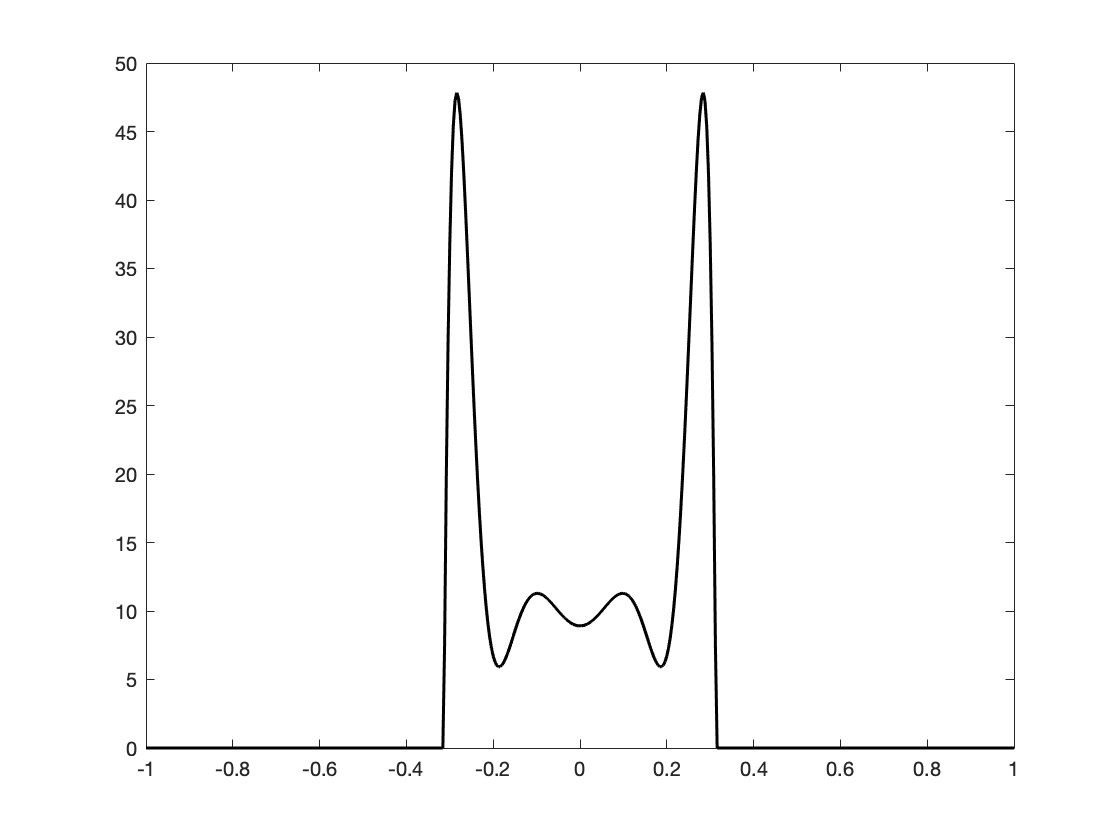}
\end{minipage}
\begin{minipage}[t]{8.4cm}
\centering
\includegraphics[scale=0.21]{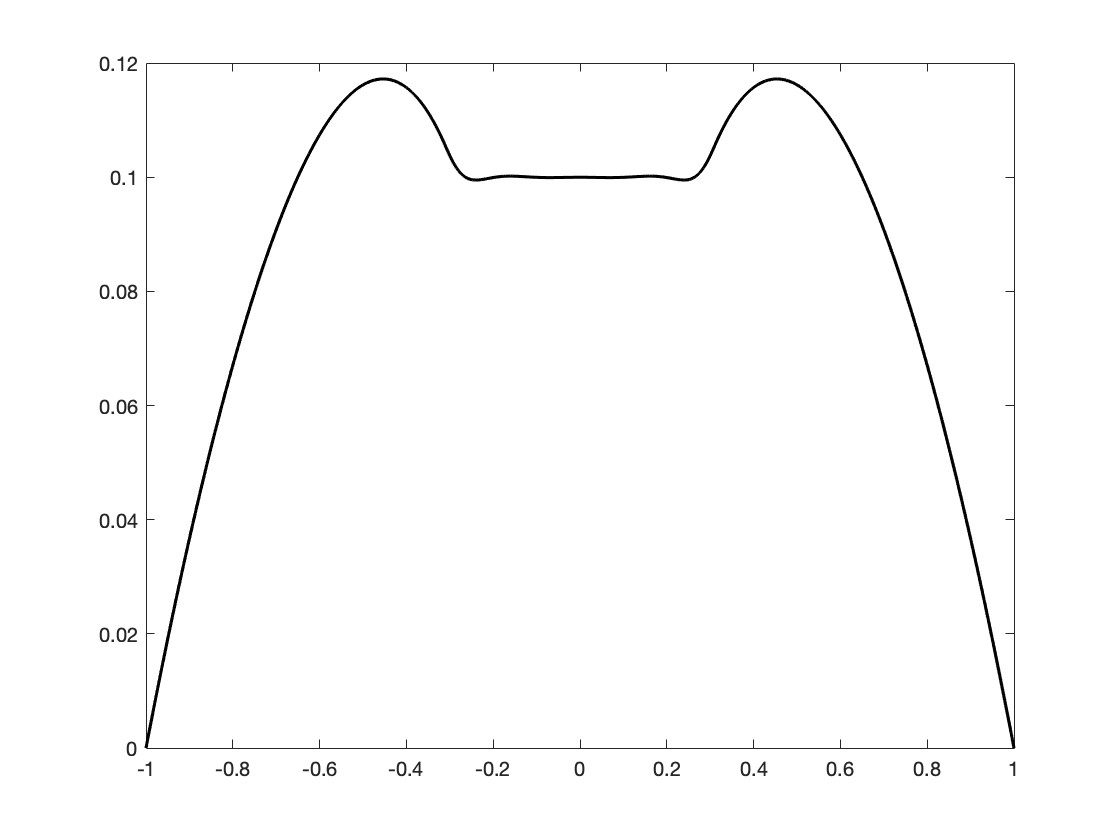}
\end{minipage}
\begin{minipage}[t]{7.4cm}
\centering
\includegraphics[scale=0.21]{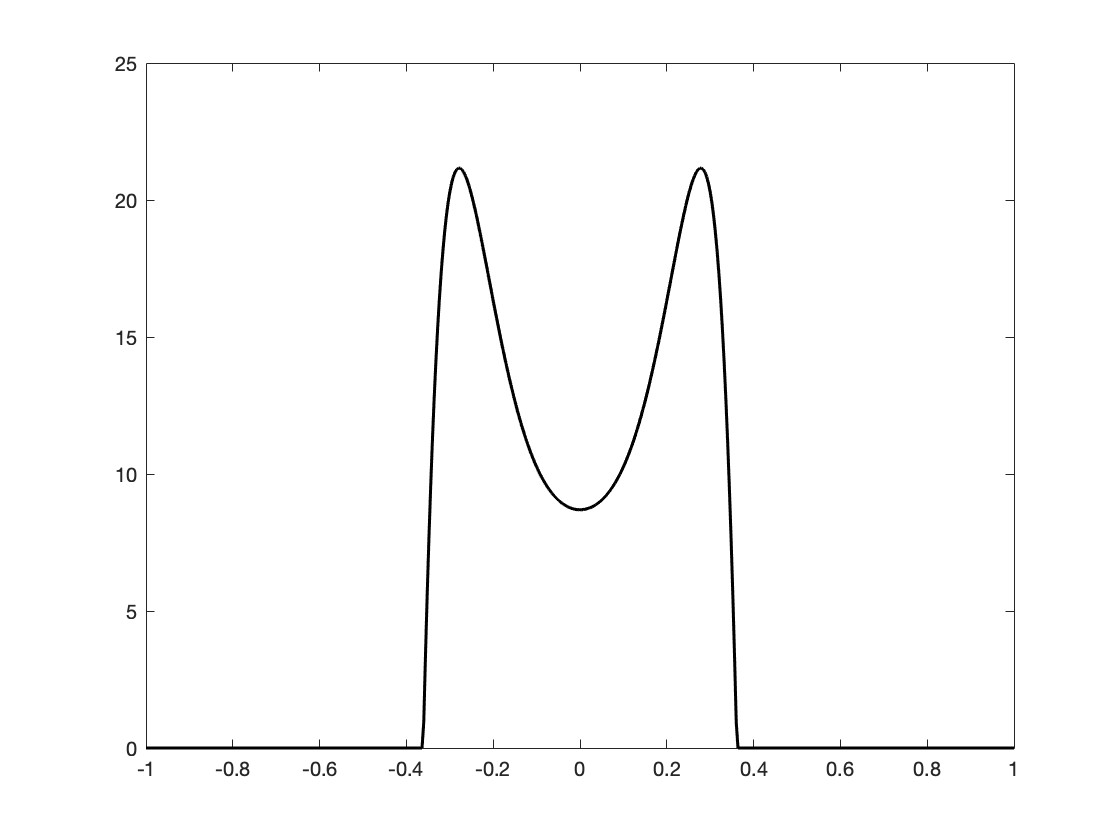}
\end{minipage}
\begin{minipage}[t]{8.4cm}
\centering
\includegraphics[scale=0.21]{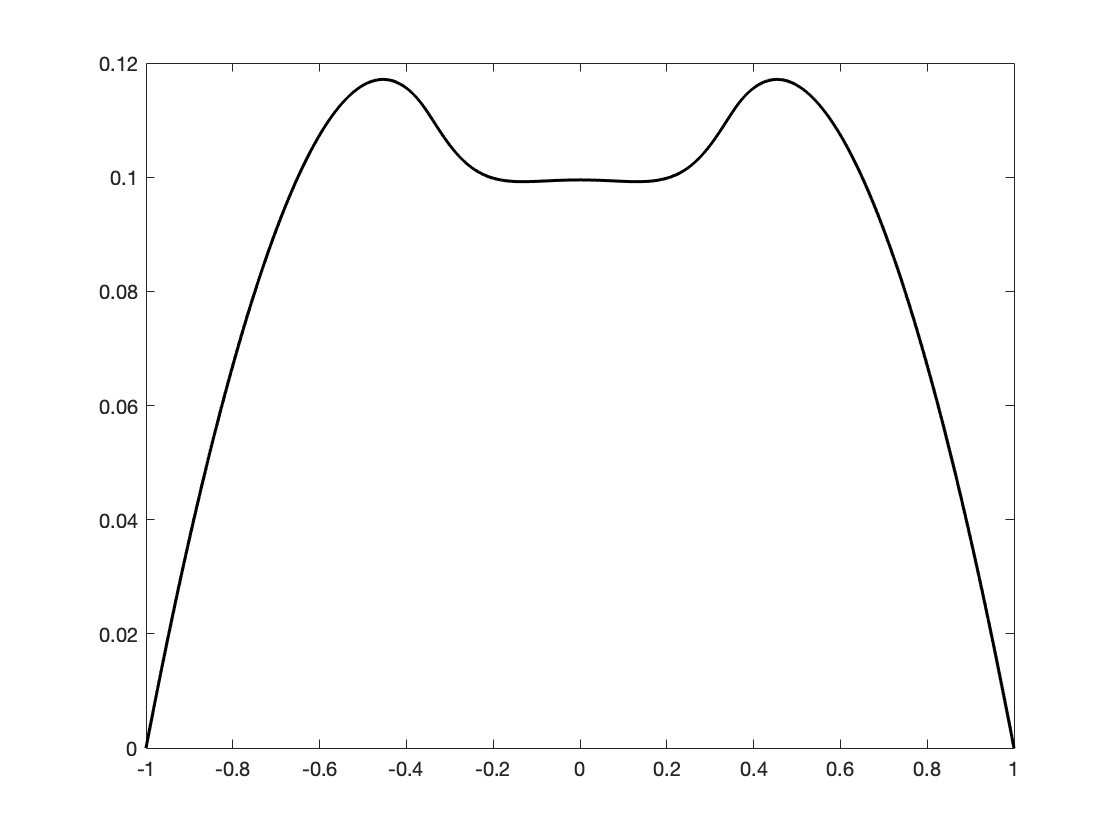}
\end{minipage}
\begin{minipage}[t]{7.4cm}
\centering
\includegraphics[scale=0.21]{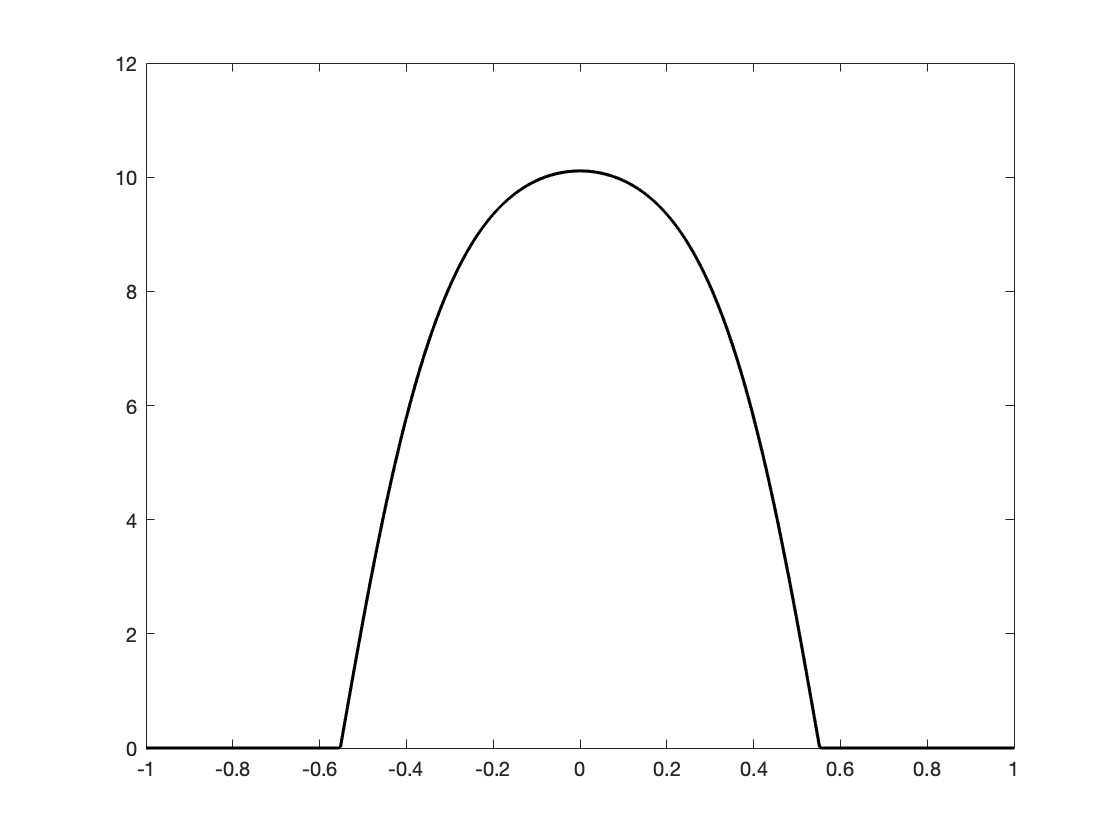}
\end{minipage}
\begin{minipage}[t]{8.4cm}
\centering
\includegraphics[scale=0.21]{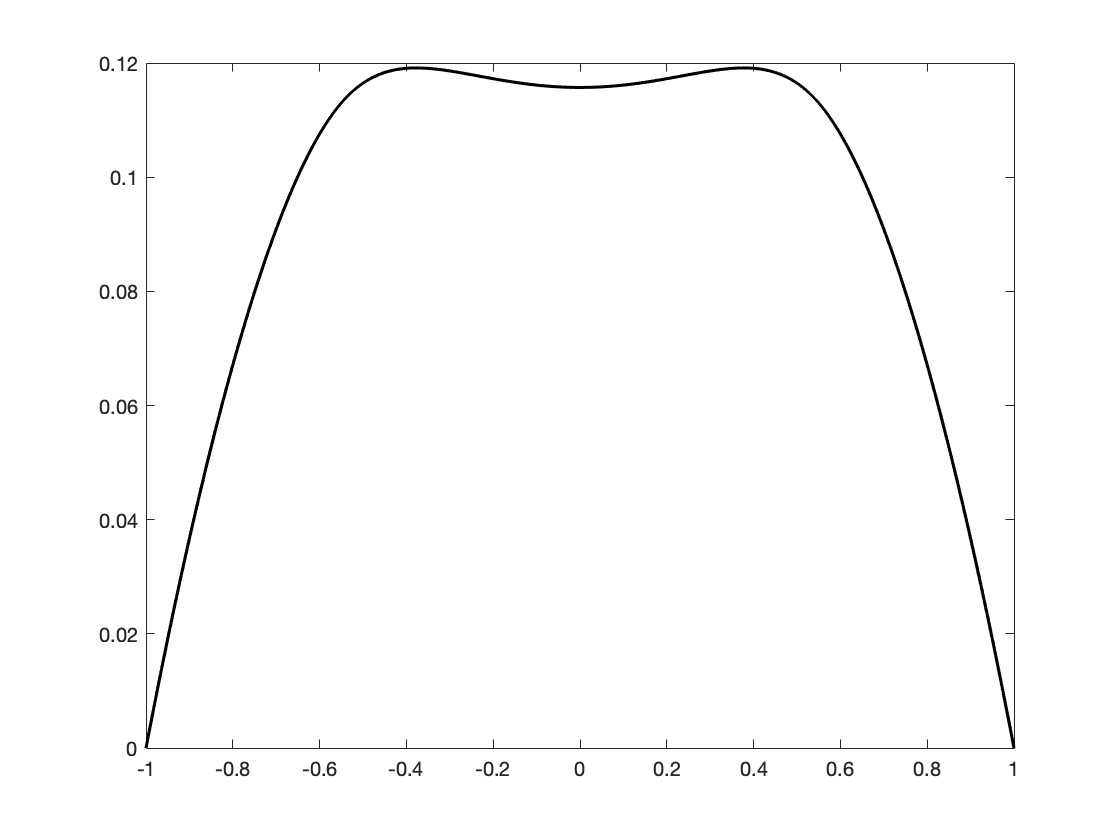}
\end{minipage}
\caption{First example: Optimal control $\hat m$ (left), and optimal state $\hat u $ (right). Top: $k=0$. Middle: $k=10^{-5}$. Bottom: $k=10^{-7}$.}
\label{Fig:Ex11}
\end{figure}
\par\bigskip\noindent
\textbf{Second example.} As a particular case of the second example in the previous section, we take
$$j(x,s)=g(x)s,\qquad\psi(s)=\infty 1_{(-\infty,\alpha)\cup(\beta,+\infty),}$$
with
$$g(x_1,x_2)=x_1^2-x_2^2.$$
The right-hand side in the state equation in (\ref{pbco}) is given by
$$f (x_1,x_2)=10(x_1^2+x_2^2)\sin\Big(13\arctan\Big|\frac{x_2}{x_1}\Big|\Big)1_{\{|x_1|>10^{-10}\}}.$$ 
Since the function $f$ changes sign several times, we expect that the corresponding optimal state function $\hat u$, solution of \eqref{ostafu}, changes sign several times too. The function $g$, which is the right-hand side of the adjoint state function $\hat z$, solution of \eqref{ostad}, also changes sign four times. Taking into account \eqref{condmu2e}, this produces a bang-bang control with $\alpha$ and $\beta$ being exchanged in several regions. For our numerical experiments we consider $\al=0$, and we take three different values for $\beta$: $1$, $10^2$ and $10^4$. The corresponding results are shown in Figure \ref{Fig:Ex2}. We observe that as $\be$ grows up, the optimal control $\hat m$ is concentrated on smaller sets. We expect that for $\be$ tending to infinity, $\hat m$ goes to a singular measure.

\begin{figure}[!http]\hspace{-1cm}
\begin{minipage}[t]{8.4cm}
\centering
\includegraphics[scale=0.14]{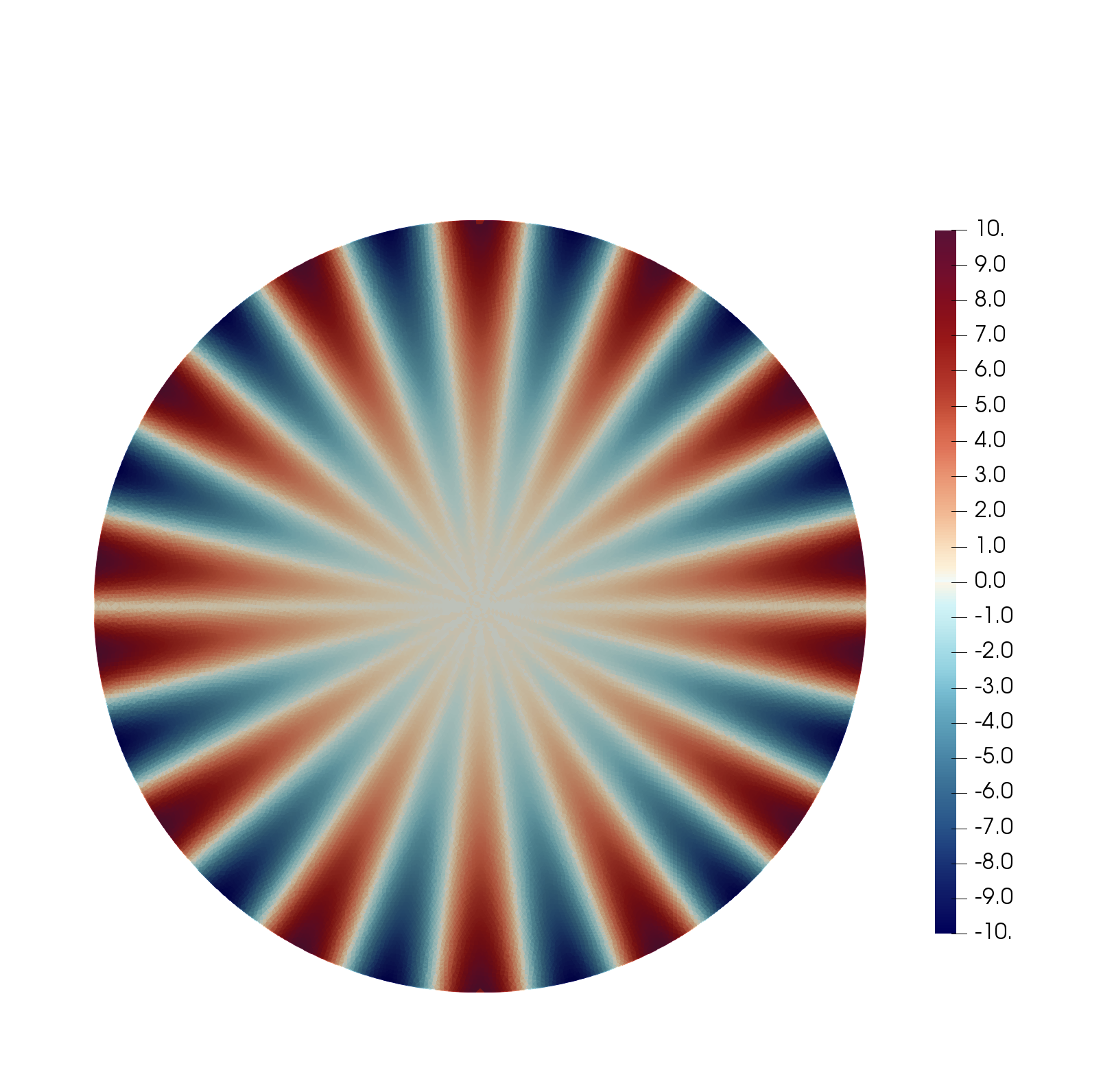}
\end{minipage}\hspace{-.5cm}
\begin{minipage}[t]{8.4cm}
\centering
\includegraphics[scale=0.14]{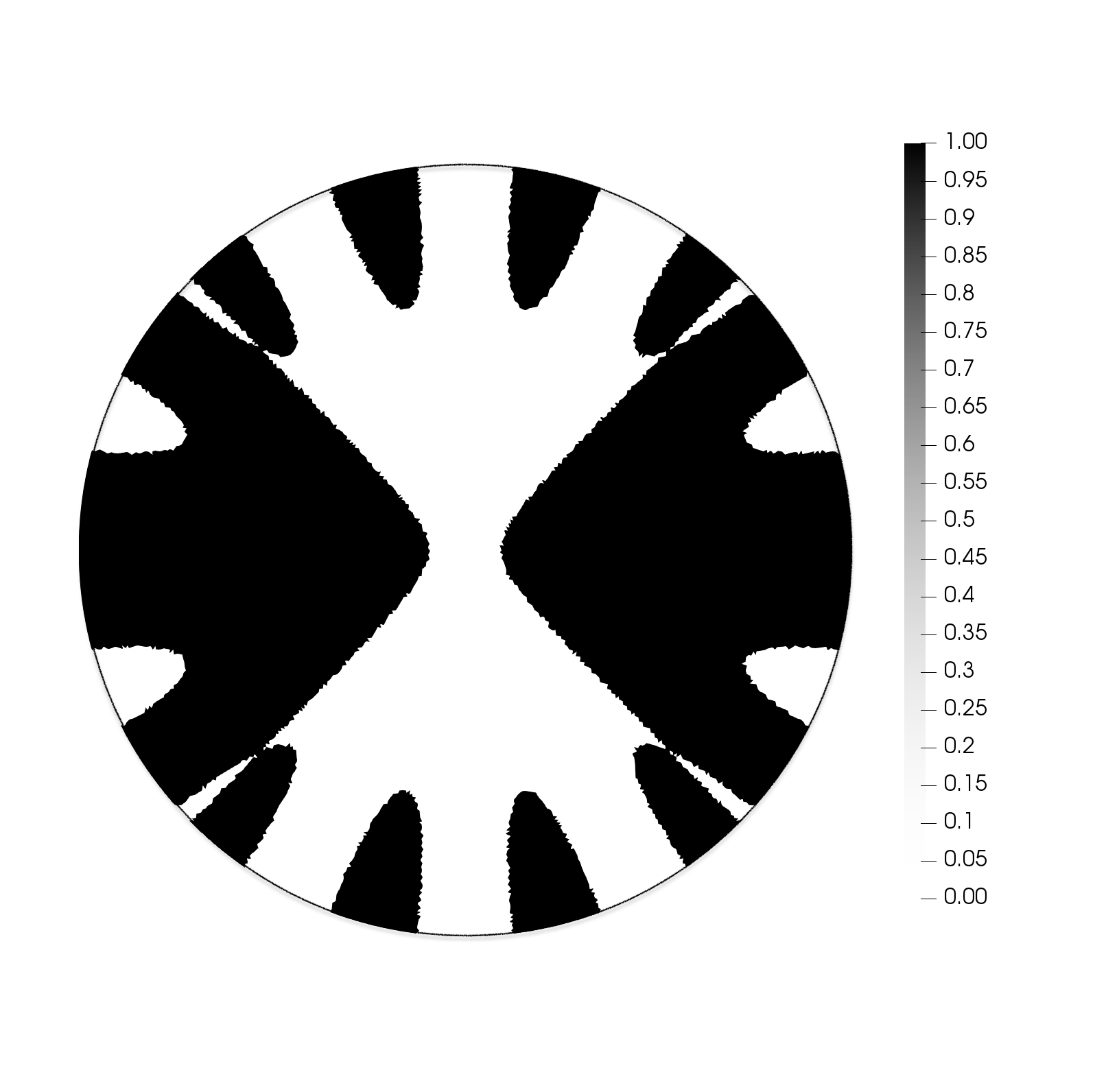}
\end{minipage}

\hspace{-1cm}
\begin{minipage}[t]{8.4cm}
\centering
\includegraphics[scale=0.14]{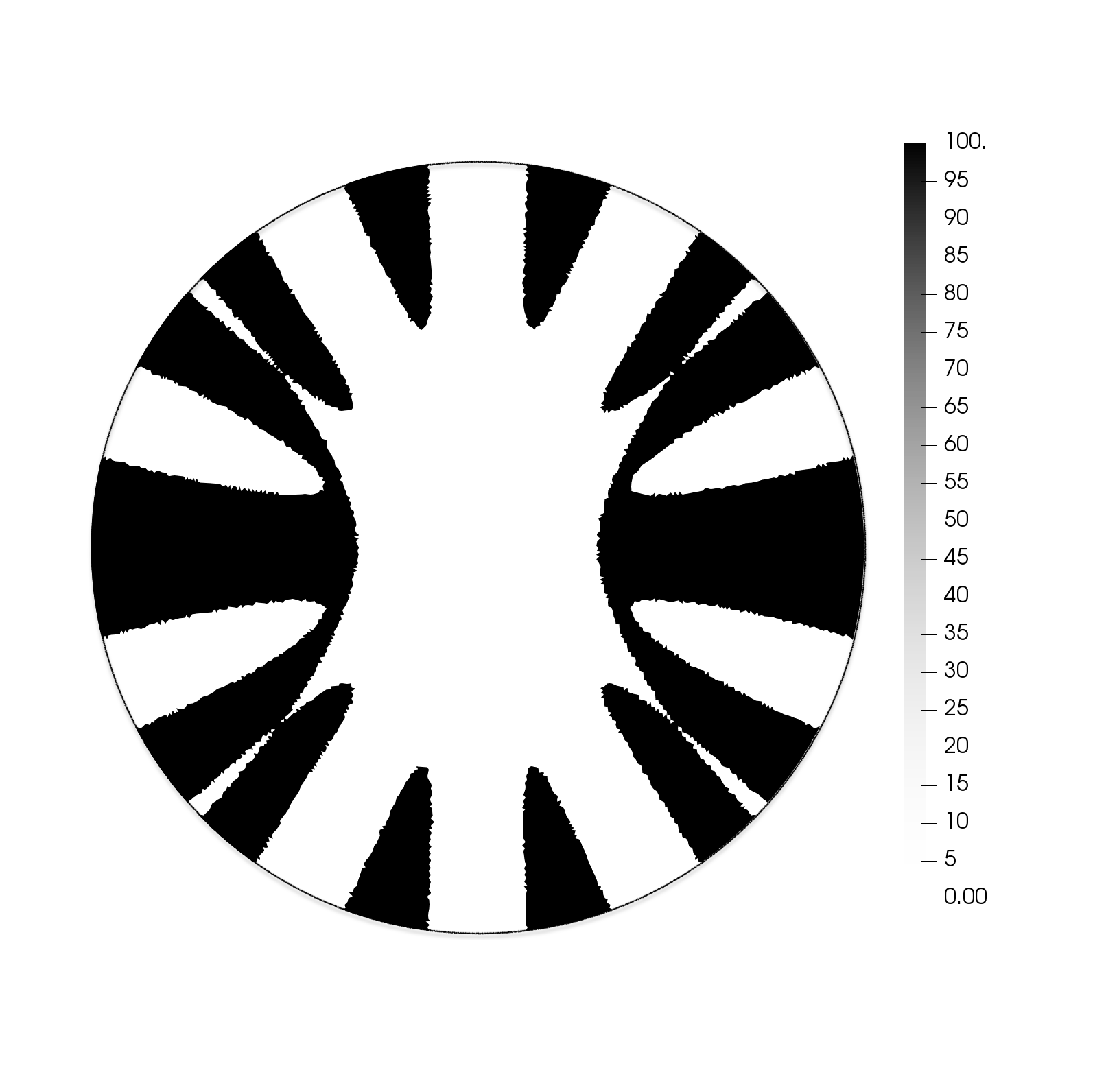}
\end{minipage}
\begin{minipage}[t]{8.4cm}
\centering
\includegraphics[scale=0.14]{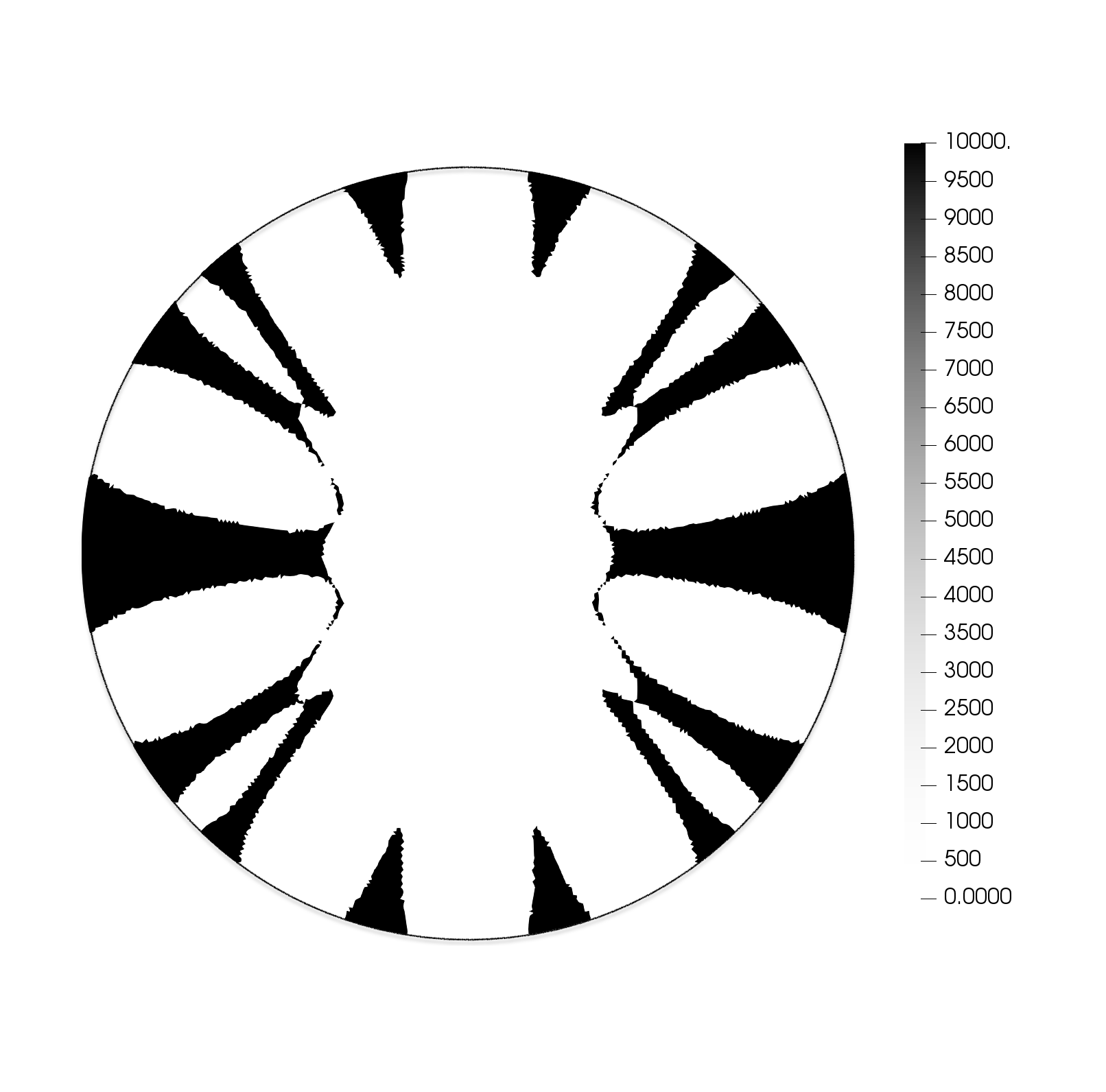}
\end{minipage}

\caption{Second example: right-hand side $f$: (top left). Optimal control $\hat m$ for $\beta=1$ (top right), $\beta=10^2$ (down left) and $\beta=10^4$ (down right).}
 \label{Fig:Ex2}
\end{figure}
\par\medskip\noindent
\textbf{Third Example.} In this case we take
$$j(x,s)=s,\qquad\psi(s)=\infty 1_{(-\infty,\alpha)\cup(\beta,+\infty)}+ks1_{[\al,\be]},$$
and
$$f(x_1,x_2)=1_{\om_1}(x_1,x_2)+1_{\om_2}(x_1,x_2)+1_{\om_3}(x_1,x_2)+1_{\om_4}(x_1,x_2).$$ 
where $\om_i=B(c_i, r)$ ($i=1,\dots,4$) are small balls of radius $r=\sqrt{3}/10$ and centers $c_1=(0,0.5)$, $c_2=(0,-0.5)$, $c_3=(0.5,0)$ and $c_4=(-0.5,0)$. Taking $\alpha=0$, $\beta=1$, we have carried out three numerical experiments corresponding to $k=0.00175$, $k=0.0014$ and $k=0.001$ respectively. The results are given in Figure \ref{Fig:Ex3}. We get a bang-bang optimal potential for the three experiments. Moreover, as expected, the norm in $L^1(B)$ of $\hat m$ increases when $k$ decreases, taking the values $0.172766$, $0.531971$ and $0.874948,$ respectively. 

\begin{figure}[!http]\hspace{-1cm}
\begin{minipage}[t]{8.4cm}
\hspace{-1.2cm}\includegraphics[scale=0.30]{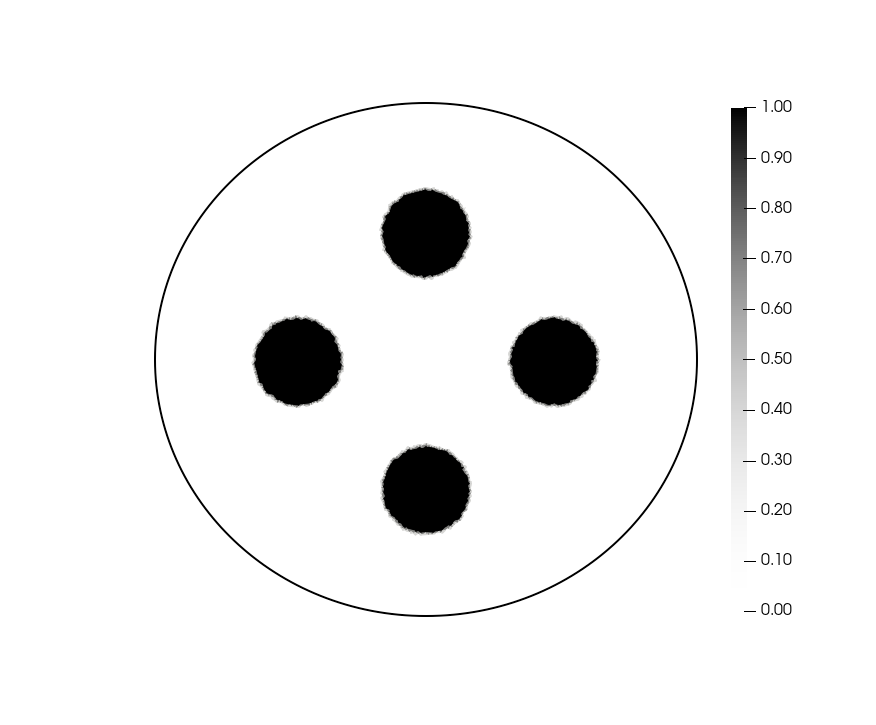}
\end{minipage}
\begin{minipage}[t]{8.4cm}
\centering
\includegraphics[scale=0.42]{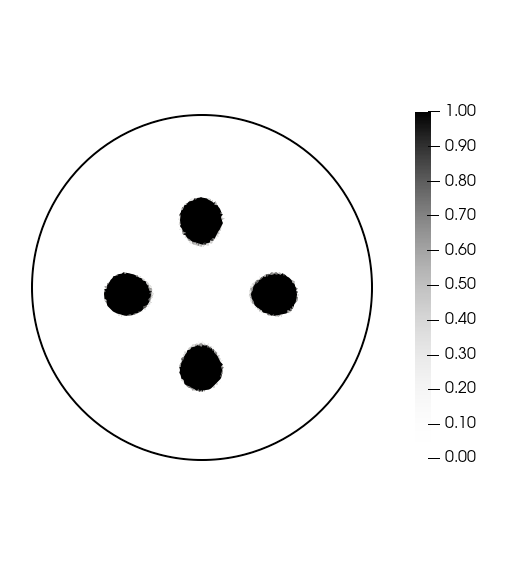}
\end{minipage}

\hspace{-1cm}
\begin{minipage}[t]{8.4cm}
\centering
\includegraphics[scale=0.42]{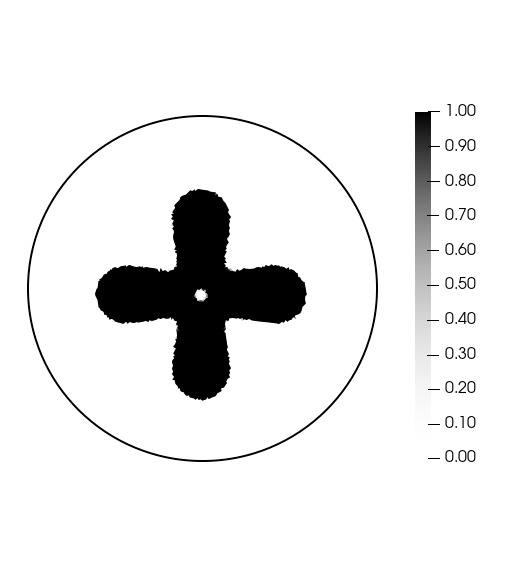}
\end{minipage}
\begin{minipage}[t]{8.4cm}
\centering
\includegraphics[scale=0.42]{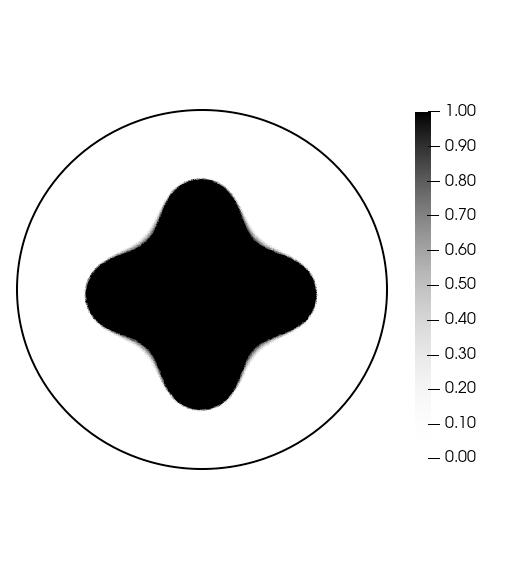}
\end{minipage}
\caption{Right-hand side $f$ (top left). Optimal control $\hat m$ for $k=0.00175$ (top right), $k=0.0014$ (down left) and $k=0.001$ (down right).}
\label{Fig:Ex3}
\end{figure}

\bigskip
\noindent{\bf Acknowledgments. }The work of GB is part of the project 2017TEXA3H {\it``Gradient flows, Optimal Transport and Metric Measure Structures''} funded by the Italian Ministry of Research and University. GB is member of the Gruppo Nazionale per l'Analisi Matematica, la Probabilit\`a e le loro Applicazioni (GNAMPA) of the Istituto Nazionale di Alta Matematica (INdAM).\par 
The work of JCD and FM is a part of the FEDER project PID2020-116809GB-I00 of the {\it Ministerio de Ciencia e Innovaci\'on} of the government of Spain and the project UAL2020-FQM-B2046 of the {\it Consejer\'{\i}a de Transformaci\'on Econ\'omica, Industria, Conocimiento y Universidades of the regional government of Andalusia, Spain.}
\bigskip

\end{document}